\documentclass[10pt]{amsart}

\usepackage{latexsym,amssymb}
\usepackage{amsmath,amsfonts,amsthm}
\input xypic

\theoremstyle{plain}
\newtheorem{theorem}{Theorem}[section]
\newtheorem{proposition}[theorem]{Proposition}
\newtheorem{corollary}[theorem]{Corollary}
\newtheorem{lemma}[theorem]{Lemma}
\newtheorem{example}[theorem]{Example}

\newtheorem{remark}[theorem]{Remark}

\newcommand{\bsl}{\setminus}

\newcommand{\dom}{\mathbf{d}}

\newcommand{\comp}{X^{\wedge} \cap Y^{\perp}}

\begin{document}

\title[Non-commutative Stone duality]{Non-commutative Stone duality:\\   
inverse semigroups, topological groupoids and $C^{\ast}$-algebras}

\author{M.~V.~Lawson}
\address{Department of Mathematics
and the
Maxwell Institute for Mathematical Sciences,
Heriot-Watt University,
Riccarton,
Edinburgh~EH14~4AS, 
Scotland}
\email{markl@ma.hw.ac.uk}

\thanks{The author's research was partially supported by an EPSRC grant (EP/F004184, EP/F014945, EP/F005881). 
He would particularly like to thank Daniel Lenz for many discussions on the subject of this paper which uses in a crucial way some of his ideas,
and also Ganna Kudryavtseva, Stuart Margolis, Pedro Resende, Philip Scott and  Benjamin Steinberg.}
 

\begin{abstract} 
We study a non-commutative generalization of Stone duality that connects a class of inverse semigroups,
called Boolean inverse $\wedge$-semigroups, with a class of topological groupoids, called Hausdorff Boolean groupoids.
Much of the paper is given over to showing that Boolean inverse $\wedge$-semigroups arise as completions of inverse semigroups we call pre-Boolean.
An inverse $\wedge$-semigroup is pre-Boolean if and only if every tight filter is an ultrafilter, where the definition of a tight filter
is obtained by combining work of both Exel and Lenz.
A simple necessary condition for a semigroup to be pre-Boolean is derived and a variety of examples
of inverse semigroups are shown to satisfy it.
Thus the polycyclic inverse monoids, and certain Rees matrix semigroups over the polycyclics,
are pre-Boolean and it is proved that the groups of units of their completions are precisely the Thompson-Higman groups $G_{n,r}$.
The inverse semigroups arising from suitable directed graphs are also pre-Boolean and
the topological groupoids arising from these graph inverse semigroups under our non-commutative Stone duality are
the groupoids that arise from the Cuntz-Krieger $C^{\ast}$-algebras.
\end{abstract}
\maketitle

%%%%%%%%%%%%%%%%%%%%%%%%%%%%%%%%%%%%%%%%%%%%%%%%%%%%%%%%%%%%%%%%%%%%%%%%%%%%%%%%%%%%%%%%%%%%%%%%%%%%%%%%%%%%%%%%%%%%%%%%%%%%%%%%%%%%%%%%%%%%%%%%%%%%%%%%%%%%%%
\section{Introduction}

From the appearance of Renault's seminal monograph \cite{Renault} and the work of Kumjian \cite{K} to the more recent book by Paterson \cite{Pat1} 
it has been known that three areas of mathematics 
\begin{center}
inverse semigroups, topological groupoids, $C^{\ast}$-algebras 
\end{center}
are closely related to each other. 
In the literature, most attention has focused on the relationship between topological groupoids and $C^{\ast}$-algebras
whereas the goal of this paper is to shift attention to that between inverse semigroups and topological groupoids. 
We prove three main theorems in this paper and explore their applications.
In the remainder of this introduction, we explain what these three theorems are  and touch on the kinds of applications we deal with.

%%%%%%%%%%%%%%%%%%%%%%%%%%%%%%%%%%%%%%%%%%%%%%%%%%%%%%%%%%%%%%%%%%%%%%%%%%%%%%%%%%%%%%%%%%%%%%%%%%%%%%%%%%%%%%%%%%%%%%%%%%%%%%%%%%%%%%%%%%%%%%
In this paper, classical Boolean algebras will be termed {\em unital Boolean algebras} whereas generalized Boolean algebras will be called simply {\em Boolean algebras}.
Thus a Boolean algebra is a distributive lattice $B$, without necessarily having a top element, 
in which for each pair of elements $a$ and $b$ there exists a, perforce unique, 
element $b \bsl a$ satisfying $a \vee b = a \vee (b \bsl a)$ and $a \wedge (b \bsl a) = 0$.
Equivalently, it is a lattice in which each principal order ideal is a unital Boolean algebra.
Homomorphisms of Boolean algebras are lattice homomorphisms preserving the bottom element.
Observe that in a non-unital Boolean algebra, the joins of all finite subsets exist,
including the emptyset whose join is 0,
and that the meets exist of all finite {\em nonempty} subsets;
for the emptyset to have a meet the algebra has to have a 1. 
A homomorphism of Boolean algebras is called {\em proper} if every element in the codomain lies beneath
an element of the image. 

A {\em Boolean space} is a Hausdorff topological space with a basis of compact-open sets.
A continuous function between topological spaces is said to be {\em proper} if the inverse image of
every compact set is also a compact set.
For background results from topology needed in this paper, see \cite{S,W}.
The following theorem is ultimately due to Marshall H. Stone \cite{Doctor}.

\begin{theorem}[Stone duality]
The category of Boolean algebras and proper homomorphisms is dual to the category of Boolean spaces and proper continuous functions.
\end{theorem}

%%%%%%%%%%%%%%%%%%%%%%%%%%%%%%%%%%%%%%%%%%%%%%%%%%%%%%%%%%%%%%%%%%%%%%%%%%%%%%%%%%%%%%%%%%%%%%%%%%%%%%%%%%%%%%%%%%%%%%%%%%%%%%%%%%%%%%%%
Let $(E,\leq)$ be a poset with zero.
If $e \in E$ such that $f \leq e$ implies that either $f = e$ or $f = 0$ then we say that $e$ is {\em $0$-minimal}.
If $X \subseteq E$ define
$$X^{\uparrow} = \{ y \in E \colon \exists x \in X; x \leq y\}$$ 
and
$$X^{\downarrow} = \{ y \in E \colon \exists x \in X; y \leq x\}.$$ 
If $X = \{x\}$ we write $x^{\uparrow}$ and $x^{\downarrow}$, respectively.
If $X = X^{\downarrow}$ we say that $X$ is {\em order ideal};
if $X$ is finite then $X^{\downarrow}$ is said to be a {\em finitely generated} order ideal.
If $X = X^{\uparrow}$ we say that $X$ is {\em closed}.
The set $X$ is said to be {\em (down) directed} if for all $x,y \in X$ there exists $z \in X$ such that $z \leq x,y$.
Observe that if $E$ is a meet semilattice then a closed set $F$ is down directed precisely when $x,y \in F$ implies that $x \wedge y \in F$.  
A subset $A \subseteq X$ is called a {\em filter} if it is directed and closed and does not contain zero.
An {\em ultrafilter} is a maximal filter.

%%%%%%%%%%%%%%%%%%%%%%%%%%%%%%%%%%%%%%%%%%%%%%%%%%%%%%%%%%%%%%%%%%%%%%%%%%%%%%%%%%%%%%%%%%%%%%%%%%%%%%%%%%%%%%%%%%%%%%%%%%%%%%%%%%%%%%%%%%%%%%%%%
Let $S$ be an inverse semigroup.
In what follows, 
the only order used in connection with inverse semigroups will be the natural partial order.
The semilattice of idempotents of $S$ is denoted by $E(S)$.
An inverse semigroup is said to be an {\em inverse $\wedge$-semigroup} if each pair of elements has a meet \cite{Leech2,Leech3}.
In an inverse semigroup, we often write $\mathbf{d}(s) = s^{-1}s$ and $\mathbf{r}(s) = ss^{-1}$.
We refer the reader to \cite{Law1} for any unproved assertions about inverse semigroups.

\begin{remark} {\em Except where stated otherwise, all inverse semigroups in this paper will be  
inverse $\wedge$-semigroups with zero. This is not a necessary condition to develop our theory
but it simplifies the mathematics --- witness Lemma~2.6(2) --- and is sufficient for the examples we have in mind.
The general theory is developed, from a different perspective, in the preprint \cite{LL}.}
\end{remark}

We say that elements $s$ and $t$ are {\em compatible}, denoted $s \sim t$, if both $s^{-1}t$ and $st^{-1}$ are idempotents.
A subset of $S$ is {\em compatible} if each pair of elements in the subset are compatible.
If $s^{-1}t = 0 = st^{-1}$ then $s$ and $t$ are said to be {\em orthogonal}.
If a pair of elements are bounded above then they are easily seen to be compatible.
It follows that for a pair of elements to be {\em eligible} to have a join they must be compatible.
Furthermore, if $s$ and $t$ are compatible then $s \wedge t$ exists and 
$\mathbf{d}(s \wedge t) = \mathbf{d}(s)\mathbf{d}(t)$, and dually.

An inverse semigroup is said to be {\em distributive} if it satisfies two conditions.
First, we require that every finite non-empty compatible subset has a join.
Second, if $\{a_{1}, \ldots, a_{m}\}$ is a non-empty finite compatible subset and if $a \in S$ is any element
then both
$\bigvee_{i=1}^{m} aa_{i}$ and   $\bigvee_{i=1}^{m} a_{i}a$ exist and we have the following two equalities
$$
a \left( \bigvee_{i=1}^{m} a_{i} \right)
=
\bigvee_{i=1}^{m} aa_{i}
\text{ and }
\left( \bigvee_{i=1}^{m} a_{i} \right) a
=
\bigvee_{i=1}^{m} a_{i}a.$$
A distributive inverse semigroup is said to be {\em Boolean} if its semilattice of idempotents is a Boolean algebra.
Distributive inverse $\wedge$-semigroups and Boolean inverse $\wedge$-semigroups will be the main classes
of inverse semigroup considered in this paper.

A {\em morphism} $\theta \colon S \rightarrow T$ of Boolean inverse $\wedge$-semigroups is a homomorphism
of inverse $\wedge$-semigroups with the property that the restriction $\theta \mid E(S) \colon E(S) \rightarrow E(T)$ is a homomorphism of Boolean algebras.
A morphism is said to be {\em proper} if the inverse images of ultrafilters in $T$ are ultrafilters in $S$.

%%%%%%%%%%%%%%%%%%%%%%%%%%%%%%%%%%%%%%%%%%%%%%%%%%%%%%%%%%%%%%%%%%%%%%%%%%%%%%%%%%%%%%%%%%%%%%%%%%%%%%%%%%%%%%%%%%%%%%%%%%%%%%%%%%%%%%%%%%%%%%%%%%%%%%%%%%%%%%
Throughout this paper categories, apart from categories of structures, will be small and objects replaced by identities.
If a category is denoted by $C$ then its set of identities will be denoted by $C_{o}$.\footnote{I follow Ehresmann's notation and write `o' for `object'.}
The elements of a category are called {\em arrows}.
Each arrow $a$ has a {\em domain}, denoted by $\mathbf{d}(a)$, and a {\em codomain} denoted by $\mathbf{r}(a)$,
both of these are identities and $a = a\mathbf{d}(a) = \mathbf{r}(a)a$.
A pair of elements $(a,b)$ in a category is {\em composable} if $\mathbf{d}(a) = \mathbf{r}(b)$.
The set of composable pairs in a category $C$ is denoted by $C \ast C$.
Given an identity $e$ the set of all arrows that begin and end at $e$ forms a monoid called the {\em local monoid} at $e$.
An arrow $a$ is {\em invertible} if there is an arrow $a^{-1}$ such that $a^{-1}a = \mathbf{d}(a)$ and $aa^{-1} = \mathbf{r}(a)$.
A category in which every arrow is invertible is called a {\em groupoid}.

Let $G$ be a groupoid.
Denote the multiplication map by $\mathbf{m}$ and the inversion map by $\mathbf{i}$.
A {\em topological groupoid} is a groupoid $G$ which is also a topological space in which
both the multiplication and inversion maps are continuous.
A topological groupoid $G$ is said to be {\em open} if the map $\mathbf{d} \colon G \rightarrow G_{o}$ is open.
If $\mathbf{d}$ is a local homeomorphism then $G$ it is said to be {\em \'etale}.
Since local homeomorphisms are open maps it follows that every \'etale groupoid is an open groupoid.
A {\em local bisection} in a groupoid $G$ is a subset $A$ such that $A^{-1}A,AA^{-1} \subseteq G_{o}$.
A {\em Boolean groupoid} is an \'etale topological groupoid with a basis of compact-open bisections
whose space of identities is a Boolean space. 
We shall be interested in this paper in {\em Hausdorff Boolean groupoids}.
A functor $\theta \colon G \rightarrow H$ between groupoids is said to be a {\em covering} functor
if it is {\em star-injective}, 
meaning that if $\theta (g) = \theta (g')$ and $\mathbf{d}(g) = \mathbf{d}(g')$ then $g = g'$,
and {\em star-surjective}, meaning that if $\mathbf{d}(h) = \theta (e)$, where $e$ is an identity, then there exists $g \in G$ such that
$\mathbf{d}(g) = e$ and $\theta (g) = h$.

%%%%%%%%%%%%%%%%%%%%%%%%%%%%%%%%%%%%%%%%%%%%%%%%%%%%%%%%%%%%%%%%%%%%%%%%%%%%%%%%%%%%%%%%%%%%%%%%%%%%%%%%%%%%%%%%%%%%%%%%%%%%%%%%%%%%%%%%%%%%%%%%%%%%%%%%%%%%%%%%%%%%%%
The key concept on which this paper is based is the {\em Lenz arrow relation} $\rightarrow$ \cite{Lenz}.
This concept is also implicit in Exel's paper \cite{E} since it is used to define the notion of a cover. 
Let $a,b \in S$.
We define $a \rightarrow b$ iff for each non-zero element $x \leq a$, we have that $x \wedge b \neq 0$.
Observe that $a \leq b \Rightarrow a \rightarrow b$.
We write $a \leftrightarrow b$ iff $a \rightarrow b$ and $b \rightarrow a$.
More generally, if $a,a_{1}, \ldots, a_{m} \in S$ then
we define $a \rightarrow \{a_{1}, \ldots, a_{m}\}$ iff for each non-zero element $x \leq a$
we have that $x \wedge a_{i} \neq 0$ for some $i$.
Finally, we write 
$$\{a_{1}, \ldots, a_{m}\} \rightarrow \{b_{1}, \ldots, b_{n}\}$$
iff $a_{i} \rightarrow  (b_{1}, \ldots, b_{n})$
for $1 \leq i \leq m$,
and we write
$$\{a_{1}, \ldots, a_{m}\} \leftrightarrow \{b_{1}, \ldots, b_{n}\}$$
iff both 
$\{a_{1}, \ldots, a_{m}) \rightarrow (b_{1}, \ldots, b_{n}\}$
and 
$\{b_{1}, \ldots, b_{n}\} \rightarrow \{a_{1}, \ldots, a_{m}\}$. 
A subset $Z \subseteq A$ is said to be a {\em cover} of $A$, 
denoted $A \rightarrow Z$, if for each $a \in A$ there exists $z \in Z$ such that $a \wedge z \neq 0$.
A special case of this definition is the following.
A finite subset $A \subseteq a^{\downarrow}$ is said to be 
a {\em cover} of $a$ if $a \rightarrow A$.\footnote{The term `cover' in this context is sanctioned by its use in frame theory. 
In the preprint \cite{LL}, we show that this notion of cover is a special case of that of a coverage which in turn is closely related to the notion of a
Grothendieck topology.}
A homomorphism $\theta \colon S \rightarrow T$ from an inverse $\wedge$-semigroup $S$ to a distributive inverse semigroup $T$ 
is said to be a {\em cover-to-join} map if for each element $s \in S$ and each finite cover $A$ of $s$
we have that $\theta (s) = \bigvee \theta (A)$.

%%%%%%%%%%%%%%%%%%%%%%%%%%%%%%%%%%%%%%%%%%%%%%%%%%%%%%%%%%%%%%%%%%%%%%%%%%%%%%%%%%%%%%%%%%%%%%%%%%%%%%%%%%%%%%%%%%%%%%%%%%%%%%%%%%%%%%%%%%%%%%%%%%%%%%%%%%%%%%%%%%%%%%%%%%
A homomorphism $\theta \colon S \rightarrow T$ between inverse semigroups with zero is said to be {\em $0$-restricted}
if $\theta (s) = 0$ implies that $s = 0$. 
Congruences that are $0$-restricted arise as kernels of such homomorphisms.
The first of our three main theorems may now be stated.
It is a wide-ranging generalization of \cite{Law7}.\\

\noindent
{\bf Completion Theorem} {\em Let $S$ be an inverse $\wedge$-semigroup. 
Then there is a distributive inverse $\wedge$-semigroup $\mathsf{D}(S)$ 
and a $0$-restricted cover-to-join homomorphism $\delta \colon S \rightarrow \mathsf{D}(S)$ having the following universal property:
for every cover-to-join map $\theta \colon S \rightarrow T$ to an arbitrary distributive inverse semigroup
there is a unique join-preserving homomorphism $\bar{\theta} \colon \mathsf{D}(S) \rightarrow T$ such that $\bar{\theta} \delta = \theta$}.\\

We explain the intuitive idea behind this theorem as follows.
Our goal is to construct from an inverse semigroup $S$ the most general distributive inverse semigroup $T$ generated by $S$ 
subject to the condition that elements of $S$ which `morally have the same join' should be identified in $T$.
The precise meaning of `morally have the same join' is encoded by the notion of cover.
Thus if $\{a_{1}, \ldots, a_{m} \} \subseteq a^{\downarrow}$ is a cover of $a$ then in our completion $T$ we shall
require that the join of the images of the $a_{i}$ be equal to the image of $a$.
The theorem says that we can indeed find such a completion of $S$.
We call $\mathsf{D}(S)$ the {\em distributive completion} of $S$.\\

Our second main theorem is as follows; it is clearly a generalization of Theorem~1.1. \\

\noindent
{\bf Duality Theorem }{\em There is a duality between the category of Boolean inverse $\wedge$-semigroups and their proper morphisms 
and the category of Hausdorff Boolean groupoids and the proper continuous covering functors between them.}

\begin{remark}{\rm The {\em monoid} version of the above theorem was proved in \cite{Law11} 
where the Hausdorff Boolean groupoids in that case have a compact space of identities.}
\end{remark}

I shall now explain how these two theorems are related to each other.
Obvious examples of Boolean inverse $\wedge$-semigroups are the symmetric inverse monoids but it is hard to think
of other examples which are not just Boolean algebras.
This raises the obvious question of finding natural examples of such semigroups.
The Completion Theorem gets us part of the way but only yields distributive inverse semigroups.
This motivates the following definition.
An inverse $\wedge$-semigroup $S$ is said to be {\em pre-Boolean} if its distributive completion $\mathsf{D}(S)$ is actually Boolean.
This definition does not of course solve anything: it simply changes the question.
However, it turns out that pre-Boolean inverse semigroups are common in mathematics
and are related to both group theory 
--- 
specifically Thompson-Higman type groups 
--- 
and the theory of $C^{\ast}$-algebras 
--- 
notably graph $C^{\ast}$-algebras.
These are discussed in Section~4 of this paper. 

All of this raises the question of how we can identify pre-Boolean inverse $\wedge$-semigroups.
This is the subject of our third main theorem.
To state it we need some notation and definitions.
Let $E$ be a meet semilattice with zero.
Let $X,Y \subseteq E$ be {\em finite} subsets.
Define 
$$X^{\wedge} = \{e \in E \colon e \leq x, \forall x \in X\},$$
the set of all elements beneath every element of $X$,
and define 
$$Y^{\perp} = \{e \in E \colon e \wedge y = 0, \forall y \in Y \},$$
the set of all elements {\em orthogonal to} every element in $Y$.
If $X$ were non-empty, we could replace $X$ by its meet, but it is convenient to have this extra flexibility.
Furthermore, if $X$ were empty we could only replace $X$ by a single element if the semilattice had a top which we do not want to assume.
If the set
$X^{\wedge} \cap Y^{\perp}$
does not consist solely of the zero we shall write $X^{\wedge} \cap Y^{\perp} \neq 0$.

A filter $F \subseteq E$ is said to be {\em tight} if for each $a \in F$ and each finite cover $\{a_{1}, \ldots, a_{m} \}$ of $a$
we have that $a_{i} \in F$ for some $i$.
We shall prove later that every ultrafilter is tight.

\begin{remark}
{\em Our definition of tight filter arises from the work of both Exel \cite{E} and Lenz \cite{Lenz}. 
We shall say more about this definition at various places in this paper.}
\end{remark}

A meet semilattice with zero $E$ satisfies the {\em trapping condition} if for all $0 \neq y < x$ the set $x^{\downarrow} \cap y^{\perp}$ has a finite cover.
The semilattice $E$ is said to be \emph{$0$-disjunctive} if for each $0 \neq f \in E$ and $e$ such that $0 \neq e < f$, 
there exists $0 \neq e' <f$ such that $e \wedge e' = 0$.
$$
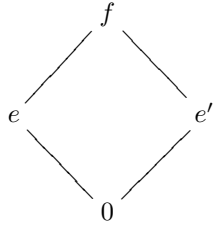
\xymatrix{
& f \ar@{-}[dl] \ar@{-}[dr] & \\
e \ar@{-}[dr] && e' \ar@{-}[dl]\\
& 0 &}
$$

\noindent
{\bf Booleanization Theorem}\mbox{}
{\em 
\begin{enumerate}

\item An inverse $\wedge$-semigroup is pre-Boolean if and only if its semilattice of idempotents is pre-Boolean.

\item Let $E$ be a meet semilattice with zero.
Then $E$ is pre-Boolean if and only if every tight filter of $E$ is an ultrafilter.

\item Let $E$ be $0$-disjunctive meet semilattice with zero. Then $E$ is pre-Boolean if and only if the trapping condition holds.
\end{enumerate}
}
The examples of pre-Boolean inverse semigroups that we shall investigate in Section~4 will satisfy condition (3) above. 

The three theorems above will be proved in Section~2.
All are new though they have varying pedigrees.
The Duality Theorem is a direct generalization of \cite{Law11} which is nothing other than the monoid version of this theorem.
Note that we have had to make some terminological changes from \cite{Law11} since the theory has outgrown the framework in which it was originally conceived.
The Booleanization Theorem shows how the work of Exel \cite{E} and Lenz \cite{Lenz} are related.
It completes the preliminary results of \cite{Law12}.
The Completion Theorem is new.

%%%%%%%%%%%%%%%%%%%%%%%%%%%%%%%%%%%%%%%%%%%%%%%%%%%%%%%%%%%%%%%%%%%%%%%%%%%%%%%%%%%%%%%%%%%%%%%%%%%%%%%%%%%%%%%%%%%%%%%%%%%%%%%%%%%%%%%%%%%%%%%%%%%%%%%%%%%
\section{Proofs of the main theorems}

In this section, we shall prove the three theorems discussed in the Introduction.
To do this, we also need to prove a technical result, called the Comparison Theorem.

%%%%%%%%%%%%%%%%%%%%%%%%%%%%%%%%%%%%%%%%%%%%%%%%%%%%%%%%%%%%%%%%%%%%%%%%%%%%%%%%%%%%%%%%%%%%%%%%%%%%%%%%%%%%%%%%%%%%%%%%%%%%%%%%%%%%%%%%%%%%%%%%%%
\subsection{The completion theorem}

We shall construct the semigroup $\mathsf{D}(S)$ in three steps.

%%%%%%%%%%%%%%%%%%%%%%%%%%%%%%%%%%%%%%%%%%%%%%%%%%%%%%%%%%%%%%%%%%%%%%%%%%%%%%%%%%%%%%%%%%%%%%%%%%%%%%%%%%%%%%%%%%%%%%%%%%%%%%%%%%%%%%%%%%%%
\begin{center}
{\bf Step~1}
\end{center}

We begin by slightly extending some results from Section~5 of \cite{Lenz}.

\begin{lemma} Let $S$ be an inverse semigroup.
\begin{enumerate}

\item The relation $\rightarrow$ is reflexive and transitive.

\item $a \rightarrow b$ if and only if $a^{-1} \rightarrow b^{-1}$.

\item If $a \rightarrow b$ and $bc = 0$ then $ac = 0$. 

\item The relation $\rightarrow$ is left and right compatible with the multiplication
in the sense that if $a \rightarrow b$ and $ac,bc \neq 0$ 
(respectively, $ca,cb \neq 0$) then $ac \rightarrow bc$
(respectively, $ca \rightarrow cb$.)

\item If $u \rightarrow s$ and $u \rightarrow t$ then $u \rightarrow s \wedge t$.

\end{enumerate}
\end{lemma}
\begin{proof}
(1) Clearly $\rightarrow$ is reflexive.
To prove transitivity, suppose that $a \rightarrow b$ and $b \rightarrow c$.
Let $0 \neq x \leq a$.
Then $x \wedge b \neq 0$.
But also $x \wedge b \leq b$ and so $x \wedge b \wedge c \neq 0$.
Hence $x \wedge c \neq 0$, as required.

(2) Straightforward.

(3) Suppose that $a \rightarrow b$ and $bc = 0$ but that $ac \neq 0$.
Then $acc^{-1} \neq 0$ and $acc^{-1} \leq a$.
Thus by assumption $b \wedge acc^{-1} \neq 0$.
But $b \wedge acc^{-1} \leq acc^{-1}$ and so $(b \wedge acc^{-1})cc^{-1} = b \wedge acc^{-1}$.
Thus $(b \wedge acc^{-1})c \neq 0$ and so $bc \wedge ac \neq 0$, which contradicts the fact that $bc = 0$.
We have proved that $bc = 0 $ implies that $ac = 0$.

(4) Suppose that $a \rightarrow b$ and $ac,bc \neq 0$.
We prove that $ac \rightarrow bc$.
Let $0 \neq x \leq ac$.
Then $x = xx^{-1}ac$ and so $xc^{-1}c = x$.
Thus $xc^{-1} \neq 0$ and $xc^{-1} \leq acc^{-1} \leq a$.
It follows that $xc^{-1} \wedge b \neq 0$.
Now $xc^{-1} \wedge b \leq xc^{-1}$ and so $(xc^{-1} \wedge b)cc^{-1} = xc^{-1} \wedge b$.
It follows that $(xc^{-1} \wedge b)c \neq 0$.
But $(xc^{-1} \wedge b)c = xc^{-1}c \wedge bc \leq x \wedge bc$.
Hence $x \wedge bc \neq 0$, as required.

(5) Let $0 \neq v \leq u$.
Then since $u \rightarrow s$ we have that $0 \neq v \wedge s$.
But $v \wedge s \leq u$ also and so since $u \rightarrow t$ we have that
$0 \neq v \wedge s \wedge t$.
We have proved that $u \rightarrow s \wedge t$.
\end{proof}

It follows from the lemma above that $\leftrightarrow$ is a $0$-restricted congruence on $S$.
We denote the $\leftrightarrow$-equivalence class containing $s$ by $\mathbf{s}$,
the quotient semigroup by $\mathbf{S}$ and the natural map $S \rightarrow \mathbf{S}$ by $\lambda$.
If $\leftrightarrow$ is just equality, we say that the inverse semigroup is {\em separative}.

\begin{lemma} 
We have that $a \rightarrow b$ if and only if $\lambda (a) \leq \lambda (b)$.
In particular, a semigroup is separative if and only if $a \leq b \Leftrightarrow a \rightarrow b$.
\end{lemma}
\begin{proof}
Suppose that $a \rightarrow b$.
We prove that $\lambda (a) \leq \lambda (b)$.
Thus we need to prove that $\lambda (a) = \lambda (ba^{-1}a)$; that is, $a \leftrightarrow ba^{-1}a$.
From $a \rightarrow b$ we get that $a^{-1} \rightarrow b^{-1}$.
Thus by compatibility, we have that $a = aa^{-1}a \rightarrow ba^{-1}a$.
It remains to show that $ba^{-1}a \rightarrow a$.
From $a \rightarrow b$ we get that $a^{-1} \rightarrow b^{-1}$.
Thus by compatibility $ba^{-1}a \rightarrow bb^{-1}a$.
But $bb^{-1}a \leq a$ implies that $bb^{-1}a \rightarrow a$.
Thus $ba^{-1}a \rightarrow a$.
Conversely, suppose that  $\lambda (a) \leq \lambda (b)$.
Then $\lambda (a) = \lambda (ba^{-1}a)$ which implies that $a \leftrightarrow ba^{-1}a$.
However, $a \rightarrow ba^{-1}a$ and $ba^{-1}a \leq b$ implies that $a \rightarrow b$.

We now prove the final claim.
Suppose that $a \leq b \Leftrightarrow a \rightarrow b$.
Then $a \leftrightarrow b$ iff $a \leq b$ and $b \leq a$ giving $a = b$.
Thus the condition implies separativity.
Conversely, suppose that the semigroup is separative and that $a \rightarrow b$.
Then $\lambda (a) \leq \lambda (b)$.
Thus $\lambda (a) = \lambda (ba^{-1}a)$.
But under the assumption of separativity this implies $a = ba^{-1}a$ and so $a \leq b$, 
as required.
\end{proof}

The semigroup $S$ is an inverse $\wedge$-semigroup.
The same is true of the quotient semigroup $\mathbf{S}$.

\begin{lemma} The semigroup $\mathbf{S}$ has all finite non-empty meets.
\end{lemma}
\begin{proof} Let $\mathbf{s}, \mathbf{t} \in \mathbf{S}$.
We shall prove that $\mathbf{s} \wedge \mathbf{t}$ exists and that
$\lambda (s) \wedge \lambda (t) = \lambda (s \wedge t)$.
The map $\lambda$ is a homomorphism of inverse semigroups and so
$\lambda (s \wedge t) \leq \lambda (s), \lambda (t)$.
Suppose that $\lambda (u) \leq \lambda (s), \lambda (t)$.
Then by Lemma~2.2, we have that $u \rightarrow s$ and $u \rightarrow t$.
By Lemma~2.1, we have that $u \rightarrow s \wedge t$.
Thus $\lambda (u) \leq \lambda (s \wedge t)$.
We have therefore proved that $\lambda (s \wedge t) = \lambda (s) \wedge \lambda (t)$.
\end{proof}

For inverse $\wedge$-semigroups, we are interested not merely in congruences but in those congruences that also preserve the meet structure.
Such congruences were first described by Leech in \cite{Leech2}.
In this paper, we shall refer to them as {\em $\wedge$-congruences}.
A congruence $\rho$ on an inverse semigroup $S$ will be called {\em separative} if $S/\rho$ is separative.

\begin{proposition} 
$\leftrightarrow$ is the smallest $0$-restricted, separative, $\wedge$-congruence on $S$.
\end{proposition}
\begin{proof}
The fact that $\leftrightarrow$ is a $0$-restricted congruence follows by Lemma~2.1
and it follows easily from the fact that $\leftrightarrow$ is $0$-restricted that it is separative.
In Lemma~2.3, we proved that it is a $\wedge$-congruence.

Now let $\rho$ be any $0$-restricted, separative $\wedge$-congruence on $S$.
Let $a \leftrightarrow b$.
Suppose there exists $0 \neq \rho (x) \leq \rho (a)$ such that $\rho (x) \wedge \rho (b) = 0$.
Then $\rho (x \wedge b) = 0$ and so since $\rho$ is $0$-restricted we have that $x \wedge b = 0$.
Now $x \wedge a \leq a$ and $x \wedge a \neq 0$ since $\rho (x) \leq \rho (a)$.
Thus by assumption $x \wedge a \wedge b \neq 0$.
It follows that $x \wedge b \neq 0$ which is a contradiction.
Thus $\rho (x) \wedge \rho (b) \neq 0$.
It follows that we have proved that $\rho (a) \rightarrow \rho (b)$.
By a symmetrical argument we get that $\rho (a)  \leftrightarrow \rho (b)$.
By assumption $\rho (a) = \rho (b)$, as required. 
\end{proof}

%%%%%%%%%%%%%%%%%%%%%%%%%%%%%%%%%%%%%%%%%%%%%%%%%%%%%%%%%%%%%%%%%%%%%%%%%%%%%%%%%%%%%%%%%%%%%%%%%%%%%%%%%%%%%%%%%%%%%%%%%%%%%%%%%%%%%%%%%%%%%%%%%%%%%%%%%%%%%%%%%%%%%
\begin{center}
{\bf Step~2}
\end{center}

The second step is a simple modification of a construction due to Boris Schein described in Section~1.4 of \cite{Law1}
and applies to any inverse semigroup.
An order ideal in an inverse semigroup is said to be {\em compatible} if it is a compatible subset.
If the order ideal is finitely generated then this is equivalent to requiring that its set of generators form a compatible subset.
We denote by $\mathsf{FC}(S)$ the set of all finitely generated compatible order ideals of $S$.
We define $\iota \colon S \rightarrow \mathsf{FC}(S)$ to be the map $s \mapsto s^{\downarrow}$.
We have the following finitary version of Theorems~1.4.23 and 1.4.24 of \cite{Law1}.

\begin{proposition} Let $S$ be an inverse semigroup with zero.
Then $\mathsf{FC}(S)$ is a distributive inverse semigroup.
If $\theta \colon S \rightarrow T$ is any homomorphism to a distributive inverse semigroup
then there is a unique join-preserving homomorphism $\theta^{\ast} \colon  \mathsf{FC}(S) \rightarrow T$
such that $\theta^{\ast} \iota = \theta$.
\end{proposition}

%%%%%%%%%%%%%%%%%%%%%%%%%%%%%%%%%%%%%%%%%%%%%%%%%%%%%%%%%%%%%%%%%%%%%%%%%%%%%%%%%%%%%%%%%%%%%%%%%%%%%%%%%%%%%%%%%%%%%%%%%%%%%%%%%%%%%%%%%%%%%%
\begin{center}
{\bf Step~3}
\end{center}

The third step is more involved.
It uses in a crucial way results to be found in \cite{Lenz} and \cite{LMS} as described in \cite{Law11}. 
We begin by constructing the groupoid $\mathsf{G}(S)$ from the ultrafilters in the inverse semigroup $S$.
The proofs of the following may be found in \cite{E}, Lemma~3.2, and \cite{Law11}, Lemma~2.5.

\begin{lemma} Let $S$ be an inverse $\wedge$-semigroup.
\begin{enumerate}

\item Every non-zero element of $S$ is contained in an ultrafilter.

\item Let $F$ be a filter in $S$.
Then $F$ is an ultrafilter if and only if $F$ contains every element $b \in S$ such that $b \wedge c \neq 0$ for all $c \in F$.

\end{enumerate}
\end{lemma}

The proofs of (1) and (2) below are Lemmas~2.6 and 2.7 of \cite{Law11}.

\begin{lemma} Let $S$ be an inverse semigroup.
\begin{enumerate}

\item If $A$ is a filter then $A = AA^{-1}A$.

\item If $A$ and $B$ are filters then $(AB)^{\uparrow}$ is the smallest filter containing $AB$.

\end{enumerate}
\end{lemma}

If $A$ and $B$ are filters we define 
$$A \cdot B = (AB)^{\uparrow}.$$
Filters have extra structure which makes them behave in a way analogous to cosets in group theory.
See Lemmas~2.8, 2.9 and 2.10 of \cite{Law11} for the proofs of (1),(2) and (3) respectively, below.

\begin{lemma} Let $S$ be an inverse semigroup.
\begin{enumerate}

\item Let $S$ be a filter. Then $B = A^{-1} \cdot A$ is a filter and inverse subsemigroup of $S$
and $A = (aB)^{\uparrow}$ for any $a \in A$.

\item Let $A$ be a filter. Then $A = A \cdot A$ iff $A$ is an inverse subsemigroup iff $A$ contains an idempotent.

\item If $A$ and $B$ are filters such that $A \cap B \neq \emptyset$ and $A^{-1} \cdot A = B^{-1} \cdot B$
(respectively, $A \cdot A^{-1} = B \cdot B^{-1}$) then $A = B$.

\end{enumerate}
\end{lemma}

Denote by $\mathsf{G}(S)$ the set of all ultrafilters of the inverse semigroup $S$.
If we restrict the definition of $A \cdot B$ to the case where $A^{-1} \cdot A = B^{-1} \cdot B$ then $\mathsf{G}(S)$ becomes a groupoid. 

We denote the set of all local bisections of the groupoid $G$ by $\mathsf{Bi}(G)$.
This is an inverse semigroup under subset multiplication.
For a proof see \cite{Pat1}, for example.

\begin{lemma} The semigroup $\mathsf{Bi}(G)$ is a Boolean inverse $\wedge$-semigroup where the natural partial order is given by inclusion 
and the idempotents are the subsets of $G_{o}$.
\end{lemma}

%%%%%%%%%%%%%%%%%%%%%%%%%%%%%%%%%%%%%%%%%%%%%%%%%%%%%%%%%%%%%%%%%%%%%%%%%%%%%%%%%%%%%%%%%%%%%%%%%%%%%%%%%%%%%%%%%%%%%%%%%%%%%%%%%%%%%%%%%%%%%%%%%%%%%%%%%%%%
For each $a \in S$, define $V_{a}$ to be the set of all ultrafilters in $S$ that contain $a$.
Put $\Omega = \Omega (S) = \{V_{a} \colon a \in S \}$.

\begin{lemma} With the above notation we have the following.
\begin{enumerate}

\item $V_{a} \cap V_{b} = V_{a \wedge b}$.

\item $V_{a}$ is a local bisection in the groupoid $\mathsf{G}(S)$.

\item $V_{a}^{-1} = V_{a^{-1}}$.

\item  $V_{a}V_{b} = V_{ab}$.

\end{enumerate}
\end{lemma}
\begin{proof}
(1) We deal with a special case first.
Suppose that $V_{a} \cap V_{b} = \emptyset$.
Then $a \wedge b = 0$ because if not there would be an ultrafilter containing $a \wedge b$ and so
an ultrafilter containing both $a$ and $b$ which contradicts the assumption that the intersection is empty.
But if $a \wedge b = 0$ then $V_{a \wedge b} = \emptyset$.
Now suppose that $a \wedge b = 0$.
Then $V_{a \wedge b} = \emptyset$.
On the other hand we must have 
$V_{a} \cap V_{b} = \emptyset$ 
because if not we would be able to show that $a \wedge b \neq 0$.

Now we can prove the general case.
Suppose that $V_{a} \cap V_{b} \neq \emptyset$.
Let $A \in V_{a} \cap V_{b}$.
Then $a,b \in A$.
But $A$ is an ultrafilter and so $a \wedge b \in A$.
It follows that $A \in V_{a \wedge b}$.
Now suppose that $a \wedge b \neq 0$.
Any ultrafilter containing $a \wedge b$ must contain both $a$ and $b$.

(2) Let $A, B \in V_{a}$ such that $A^{-1} \cdot A = B^{-1} \cdot B$.
By assumption $A \cap B \neq \emptyset$.
Thus by Lemma~2.8(3), we have that $A = B$.
A similar argument works dually.

(3) This follows from the observation that $A$ is an ultrafilter iff $A^{-1}$ is an ultrafilter.

(4) Let $A \in V_{a}$ and $B \in V_{b}$.
Then $a \in A$ and $b \in B$ so $ab \in A \cdot B$ and $A \cdot B \in V_{ab}$.

Now let $C \in V_{ab}$.
Thus $ab \in C$.
Put $H = C^{-1} \cdot C$ so that $C = (ab H)^{\uparrow}$.
We have that $(ab)^{-1}ab \in H$ but $b^{-1}a^{-1}ab \leq b^{-1}b$ and so $b^{-1}b \in H$.
It follows that $B = (bH)^{\uparrow}$ is a well-defined ultrafilter containing $b$.
Thus $B \in V_{b}$.
Observe that $B \cdot B^{-1} = (bHb^{-1})^{\uparrow}$.
Now  $(ab)^{-1}ab \in H$ and so $bb^{-1}a^{-1}abb^{-1} \in (bHb^{-1})^{\uparrow}$.
Thus $a^{-1}abb^{-1} \in (bHb^{-1})^{\uparrow}$ from which it follows that $a^{-1}a \in (bHb^{-1})^{\uparrow}$.
It follows that $A = (a(bHb^{-1})^{\uparrow})^{\uparrow}$ is a well-defined ultrafilter containing $a$.
Observe that $A^{-1} \cdot A = B \cdot B^{-1}$ and that $A \cdot B = C$.
\end{proof}

The following result is in part from \cite{Lenz} and links the Lenz arrow with ultrafilters.

\begin{lemma} Let $S$ be an inverse semigroup.
\begin{enumerate}

\item Let $a,b \in S$.
Then $V_{a} \subseteq V_{b}$ if and only if $a \rightarrow b$.

\item Let $a, a_{1}, \ldots, a_{m} \in S$.
Then $a \rightarrow (a_{1}, \ldots, a_{m})$ if and only if $V_{a} \subseteq \bigcup_{i=1}^{m} V_{a_{i}}$.

\item Let  $a_{1}, \ldots, a_{m}, b_{1}, \ldots, b_{n} \in S$.
Then $(a_{1}, \ldots, a_{m}) \leftrightarrow (b_{1}, \ldots, b_{n})$ if and only if
$\bigcup_{i=1}^{m} V_{a_{i}} = \bigcup_{j=1}^{m} V_{b_{j}}$.

\item Let $a \in S$ be such that every ultrafilter in $V_{a}$ is idempotent.
Then $a \leftrightarrow a \wedge a^{-1}a$.

\end{enumerate}
\end{lemma}
\begin{proof} We prove (1); the proof of (2) is an immediate generalization, and (3) is immediate from (2). 
Suppose that $V_{a} \subseteq V_{b}$.
Let $0 \neq x \leq a$.
By Lemma~2.6(1), there is an ultrafilter $F$ containing $x$.
But then $a \in F$.
By assumption $b \in F$.
But $x,b \in F$, where $F$ is an ultrafilter, implies that $x \wedge b \neq 0$.

To prove the converse, suppose that for each non-zero element $x \leq a$, we have that $x \wedge b \neq 0$.
Let $F \in V_{a}$.
Thus $a \in F$.
Suppose that $b \notin F$.
Then by Lemma~2.6(2), there exists $y \in F$ such that $b \wedge y = 0$.
Now $a,y \in F$ implies that $x = a \wedge y \neq 0$
and by construction $x \leq a$.
Thus by our assumption, $x \wedge b \neq 0$.
But $x \wedge b = a \wedge y \wedge b \leq b \wedge y = 0$,
which is a contradiction.
Thus $b \in F$ and so  $V_{a} \subseteq V_{b}$.

(4) Put $e_{a} = a \wedge a^{-1}a$.
Suppose that every ultrafilter in $V_{a}$ is idempotent.
By Lemma~2.8(2), this is equivalent to saying that every such ultrafilter is an inverse subsemigroup.
It follows that if $F \in V_{a}$ then from $a \in F$ we get that $a^{-1} \in F$ and so $a^{-1}a  \in F$.
Hence $a \wedge a^{-1}a \in F$.
We have show that $V_{a} \subseteq V_{e_{a}}$.
On the other hand, $e_{a} \leq a$ and so $V_{e_{a}} \subseteq V_{a}$.
It follows that $V_{a} = V_{e_{a}}$ and so $a \leftrightarrow a \wedge a^{-1}a = e_{a}$.
\end{proof}

%%%%%%%%%%%%%%%%%%%%%%%%%%%%%%%%%%%%%%%%%%%%%%%%%%%%%%%%%%%%%%%%%%%%%%%%%%%%%%%%%%%%%%%%%%%%%%%%%%%%%%%%%%%%%%%%%%%%%%%%%%%%%%%%%%%%%%%%%%%%%%%%%%
\begin{lemma} Let $A = \{a_{1}, \ldots, a_{n} \}^{\downarrow}$ be a finitely generated compatible order ideal. 
Then
$$\bigcup_{a \in A} V_{a} = \bigcup_{i=1}^{n} V_{a_{i}}.$$
\end{lemma}
\begin{proof}
Only one direction needs proving.
Let $F \in \bigcup_{a \in A} V_{a}$.
Then $F \in V_{a}$ for some $a \in A$.
But $a \leq a_{i}$ for some $i$ and so $F \in V_{a_{i}}$, as required. 
\end{proof}

\begin{lemma} 
If $a$ and $b$ are compatible then $V_{a} \cup V_{b}$ is a bisection.
\end{lemma}
\begin{proof} Suppose that $F \in V_{a}$ and $G \in V_{b}$ and that $F^{-1} \cdot F = G^{-1} \cdot G$.
Then $F \cdot G^{-1}$ is defined and contains the idempotent $ab^{-1}$.
It follows by Lemma~2.8(2) that  $F \cdot G^{-1}$ is an idempotent ultrafilter and so is an identity in the groupoid.
Thus $F = G$, as required.
The dual result follows by symmetry.
\end{proof}

%%%%%%%%%%%%%%%%%%%%%%%%%%%%%%%%%%%%%%%%%%%%%%%%%%%%%%%%%%%%%%%%%%%%%%%%%%%%%%%%%%%%%%%%%%%%%%%%%%%%%%%%%%%%%%%%%%%%%%%%%%%%%%%%%%%%%%%%%%%%%%%%%%%%%%%%
Let $S$ be an inverse $\wedge$-semigroup.
Then $\mathsf{G}(S)$ is the groupoid of ultrafilters of $S$.
Observe that here this groupoid is considered without a topology.
By Lemma~2.9, $\mathsf{Bi}(\mathsf{G}(S))$ is the Boolean inverse $\wedge$-semigroup of all local bisections of $\mathsf{G}(S)$. 
By Lemmas~2.12 and 2.13, 
the map
$$\beta \colon \mathsf{FC}(S) \rightarrow \mathsf{Bi}(\mathsf{G}(S))
\text{ defined by }
\beta (A) = \bigcup_{i=1}^{n} V_{a_{i}}$$ 
where $A = \{a_{1}, \ldots, a_{n} \}^{\downarrow}$
is well-defined.
If  $A = \{a_{1}, \ldots, a_{m} \}^{\downarrow}$ and $B = \{b_{1}, \ldots, b_{n} \}^{\downarrow}$
are elements of $\mathsf{FC}(S)$ define
$$A \equiv B \Leftrightarrow (a_{1}, \ldots, a_{m}) \leftrightarrow (b_{1}, \ldots, b_{n}).$$

\begin{lemma} The map $\beta$ is a homomorphism.
In addition, 
if $A = \{a_{1}, \ldots, a_{m} \}^{\downarrow}$ and $B = \{b_{1}, \ldots, b_{n} \}^{\downarrow}$
then 
$$\beta (A) = \beta (B) \Leftrightarrow A \equiv B.$$
\end{lemma}
\begin{proof}
Let  $A = \{a_{1}, \ldots, a_{m} \}^{\downarrow}$ and $B = \{b_{1}, \ldots, b_{n} \}^{\downarrow}$.
Then
$$\beta (A) = \bigcup_{i=1}^{m} V_{a_{i}}   
\text{ and }
\beta (B) = \bigcup_{j=1}^{n} V_{b_{j}}.$$  
Multiplying the two unions together and then using the fact that $\mathsf{Bi}(\mathsf{G}(S))$ is a Boolean inverse $\wedge$-semigroup by Lemma~2.9
together with Lemma~2.10 we get that
$$\beta (A)\beta (B) =  \bigcup_{i,j} V_{a_{i}b_{j}}.$$  
But it is easy to check that
$AB = \{a_{i}b_{j} \colon 1 \leq i \leq m, 1 \leq j \leq n \}^{\downarrow}$
and so $\beta (A)\beta (B) = \beta (AB)$.
The final claim follows by Lemma~2.11(3).
\end{proof}

It follows that the kernel of the homomorphism $\beta$ is the congruence $\equiv$.
We denote the $\equiv$-class containing $A$ by $[A]$.
We denote the natural map from $\mathsf{FC}(S)$ to $\mathsf{FC}(S)/\equiv$ by $\xi$.
We denote by $\beta'$ the unique map such that $\beta' \xi = \beta$.
The map $\beta'$ is, of course, injective and the image of $\beta'$ is the same as the image of $\beta$.
$$
\xymatrix{
S \ar[r]^{\iota} &\mathsf{FC}(S) \ar@{->>}[r]^{\xi}  \ar[d]_{\beta} & \mathsf{FC}/\equiv  \ar@{^(->}[dl]^{\beta'} \\
& \mathsf{Bi}(\mathsf{G}(S)) & \\
}
$$

%%%%%%%%%%%%%%%%%%%%%%%%%%%%%%%%%%%%%%%%%%%%%%%%%%%%%%%%%%%%%%%%%%%%%%%%%%%%%%%%%%%%%%%%%%%%%%%%%%%%%%%%%%%%%%%%%%%%%%%%%%%%%%%%%%%%%%%%%%%%%%%%%%%%%%%%%%%
\begin{center}
{\bf Restriction to the case where $S$ is separative}
\end{center}

{\em At this point, we shall restrict the class of inverse semigroups we consider.
We shall, however, return to the general case later.}
Let $S$ be a {\em separative} semigroup.
From Lemmas~2.2 and 2.11(1), this means that for non-zero elements $a,b \in S$ we have that 
$a = b$ if and only if  $a \leftrightarrow b$ if and only if $V_{a} = V_{b}$.
Define 
$$\mathsf{D}(S) = \mathsf{FC}(S)/\equiv
\text{ and }
\delta (s) = [s^{\downarrow}].\footnote{The definition of $\mathsf{D}(S)$ for arbitrary $S$ will be given below.}$$
Thus $\delta = \xi \iota$.
We shall write $\delta_{S}$ rather than $\delta$ when there is a danger of ambiguity.
The inverse semigroup $\mathsf{D}(S)$ is embedded in $\mathsf{Bi}(\mathsf{G}(S))$ by $\beta'$.
We now describe the image of $\beta'$.

\begin{lemma} Let $S$ be separative.
\begin{enumerate}

\item $\bigcup_{i=1}^{m} V_{a_{i}}$ is a bisection in $\mathsf{Bi}(\mathsf{G}(S))$
if and only if $\{a_{1}, \ldots, a_{m} \}$ is a compatible subset of $S$.

\item The image of $\beta$ consists precisely of all bisections in $\mathsf{Bi}(\mathsf{G}(S))$ of the form $\bigcup_{i=1}^{m} V_{a_{i}}$
where  $\{a_{1}, \ldots, a_{m} \}$ is a compatible subset of $S$.

\end{enumerate}
\end{lemma}
\begin{proof} (1) It follows by Lemma~2.13 that if $\{a_{1}, \ldots, a_{m} \}$ is a compatible subset of $S$
then $\bigcup_{i=1}^{m} V_{a_{i}}$ is a bisection in $\mathsf{Bi}(\mathsf{G}(S))$.
We prove the converse.
Suppose that $\bigcup_{i=1}^{m} V_{a_{i}}$ is a bisection in $\mathsf{Bi}(\mathsf{G}(S))$.
Then using Lemmas~2.9 and 2.10, the sets $V_{a_{i}^{-1}a_{j}}$ and $V_{a_{i}a_{j}^{-1}}$
are either empty if $a_{i}^{-1}a_{j} = 0$ or $a_{i}a_{j}^{-1} = 0$ or consist of identities in the groupoid that
is, by Lemma~2.8(2), ultrafilters that are also inverse subsemigroups.
It follows by Lemma~2.11(4) that if $V_{x}$ consists of only idempotent ultrafilters then
$x \leftrightarrow x \wedge x^{-1}x$.
But $S$ is separative and so $x$ is an idempotent.
Hence in all cases $a_{i}^{-1}a_{j}$ and $a_{i}a_{j}^{-1}$ are idempotents and so $a_{i}$ and $a_{j}$ are compatible.

(2) This is now immediate.
\end{proof}

\begin{proposition} Let $S$ be separative inverse $\wedge$-semigroup.
Let $\theta \colon S \rightarrow T$ be a cover-to-join homomorphism to the distributive inverse semigroup $T$.
Then there exists a unique join-preserving homomorphism $\bar{\theta} \colon \mathsf{D}(S) \rightarrow T$
such that $\bar{\theta} \delta = \theta$.
\end{proposition}
\begin{proof} By Proposition~2.5, there exists  a unique join-preserving
homomorphism $\theta^{\ast} \colon \mathsf{FC}(S) \rightarrow T$ such that $\theta^{\ast} \iota = \theta$.
By definition
$$\theta^{\ast}(\{a_{1}, \ldots, a_{m} \}^{\downarrow}) = \bigvee_{i=1}^{m} \theta (a_{i}).$$

We now prove that if
$(a_{1}, \ldots, a_{m}) \leftrightarrow (b_{1}, \ldots, b_{n})$ then
$\bigvee_{i} \theta (a_{i}) = \bigvee_{j} \theta (b_{j})$.
By definition, for each $i$ we have that
$a_{i} \rightarrow (b_{1}, \ldots, b_{n})$.
It is easy to check that
$\{a_{i} \wedge b_{1}, \ldots, a_{i} \wedge b_{n} \} \subseteq a_{i}^{\downarrow}$ is a cover.
But $\theta$ is a cover-to-join map and so for each $i$ we have that
$\theta (a_{i}) = \bigvee_{j} \theta (a_{i} \wedge b_{j})$.
It follows that
$$\bigvee_{i=1}^{m} \theta (a_{i}) = \bigvee_{i,j} \theta (a_{i} \wedge b_{j}).$$
But a similar result holds by symmetry for $\bigvee_{j} \theta (b_{j})$ and we have proved our claim.

It follows that we may define $\bar{\theta}([\{a_{1}, \ldots, a_{m} \}^{\downarrow}]) = \bigvee_{i} \theta (a_{i})$
and so have a well-defined function $\bar{\theta} \colon \mathsf{D}(S) \rightarrow T$
such that $\bar{\theta} \xi = \theta^{\ast}$.
Thus $\bar{\theta} \delta = \theta$.
It is clear that it is a homomorphism and that it is join-preserving.
Uniqueness follows almost immediately.
 \end{proof}

%%%%%%%%%%%%%%%%%%%%%%%%%%%%%%%%%%%%%%%%%%%%%%%%%%%%%%%%%%%%%%%%%%%%%%%%%%%%%%%%%%%%%%%%%%%%%%%%%%%%%%%%%%%%%%%%%%%%%%%%%%%%%%%%%%%%%%%%%%%%%%
\begin{center}
{\bf Proof of the Completion Theorem}
\end{center}

{\bf We now return to the general case.
Let $S$ be an arbitrary inverse semigroup.}
Put $\mathbf{S} = S/\leftrightarrow$, a separative semigroup by Proposition~2.4.
Recall that the $\leftrightarrow$-equivalence class containing $s$ is denoted by $\mathbf{s}$.
Define 
$$\mathsf{D}(S) = \mathsf{FC}(\mathbf{S})/\equiv
\text{ and }
\delta (s) = [\mathbf{s}^{\downarrow}].\footnote{This is the definition of $\mathsf{D}(S)$ for arbitrary $S$ not just those $S$ which are separative.}$$
Observe that $\delta$ is $0$-restrictive since it is essentially the map $s \mapsto V_{s}$ and by Lemma~2.6(1) this is non-empty if $s$ is non-zero.
Furthermore, $\mathsf{D}(S)$ is an inverse $\wedge$-semigroup essentially by Lemma~2.10.

\begin{proof}
Let $\theta \colon S \rightarrow T$ be a cover-to-join map to a distributive inverse semigroup $T$.

The first step in the proof is to show that if $a,b \in S$ are non-zero elements such that
$a \leftrightarrow b$ then $\theta (a) = \theta (b)$.
Observe first that because $0 \neq a \leq a$ we have that $a \wedge b \neq 0$.
We prove now that $a \rightarrow a \wedge b$.
The fact that also $b \rightarrow a \wedge b$ will follow by symmetry.
Let $0 \neq x \leq a$.
Then $x \wedge b \neq 0$.
But $x = x \wedge a$ and so $x \wedge a \wedge b \neq 0$, as required.
Hence $\{a \wedge b \}$ is a cover of both $a$ and $b$.
But $\theta$ is a cover-to-join map and so $\theta (a) = \theta (a \wedge b) = \theta (b)$, as required.
It follows that there is a homomorphism $\phi \colon \mathbf{S} \rightarrow T$ such that $\phi \lambda = \theta$.

It remains to show that $\phi$ is a cover-to-join map.
Let $\{\mathbf{a}_{1}, \ldots, \mathbf{a}_{n} \} \subseteq \mathbf{a}^{\downarrow}$ be a cover.
Then by Lemma~2.2, $a_{i} \rightarrow a$ for each $i$ in the semigroup $S$.
We claim that $\{a \wedge a_{i} \colon 1 \leq i \leq n \} \subseteq a^{\downarrow}$ is a cover.
Let $0 \neq x \leq a$.
Then $\mathbf{0} \neq \mathbf{x} \leq \mathbf{a}$ in $\mathbf{S}$ using the fact that $\leftrightarrow$ is $0$-restricted.
It follows that $\mathbf{x} \wedge \mathbf{a}_{i} \neq \mathbf{0}$.
However, by Proposition~2.4 the congruence $\leftrightarrow$ is a $\wedge$-congruence.
This gives us that $x \wedge a_{i} \neq 0$.
It follows that $x \wedge a \wedge a_{i} \neq 0$.
Since $\theta$ is a cover-to-join map, we have that
$$\theta (a) = \bigvee_{i=1}^{n} \theta (a \wedge a_{i}).$$
However, $\mathbf{a}_{i} \leq \mathbf{a}$ and so $\mathbf{a} \wedge \mathbf{a}_{i} = \mathbf{a}_{i}$.
It follows that $a \wedge a_{i} \leftrightarrow a_{i}$.
Hence $\theta (a_{i} \wedge a) = \theta (a_{i})$.
So we have that  
$$\theta (a) = \bigvee_{i=1}^{n} \theta (a_{i}).$$
But $\phi \lambda = \theta$ and so
$$\theta (\mathbf{a}) = \bigvee_{i=1}^{n} \theta (\mathbf{a}_{i})$$
as required.

The homomorphism $\phi \colon \mathbf{S} \rightarrow T$ is a cover-to-join map to a finitely complete distributive inverse semigroup.
Thus by Proposition~2.16, there is a unique join-preserving homomorphism
$\bar{\phi} \colon \mathsf{D}(\mathsf{\mathbf{S}}) \rightarrow T$ such that $\bar{\phi} \delta_{\mathbf{S}} = \phi$.
$$\xymatrix{
S \ar[r]^{\lambda} \ar[dr]_{\theta}  & \mathbf{S} \ar[d]^{\phi} \ar[r]^{\delta_{\mathbf{S}}} & \mathsf{D}(\mathbf{S}) \ar[dl]^{\bar{\phi}}\\
& T &
}
$$
The map $\delta \colon S \rightarrow \mathsf{D}(S)$ is just $\delta_{\mathbf{S}}\lambda$ and we shall rename $\bar{\phi}$ as $\bar{\theta}$.
We then have that $\bar{\theta} \delta = \theta$.
Uniqueness is almost immediate.
\end{proof}

%%%%%%%%%%%%%%%%%%%%%%%%%%%%%%%%%%%%%%%%%%%%%%%%%%%%%%%%%%%%%%%%%%%%%%%%%%%%%%%%%%%%%%%%%%%%%%%%%%%%%%%%%%%%%%%%%%%%%%%%%%%%%%%%%%%%%%%%%
\subsection{The Booleanization Theorem}

In the previous section, we constructed a distributive inverse $\wedge$-semigroup from every inverse $\wedge$-semigroup $S$.
The goal of this section is to find conditions on $S$ that imply that $\mathsf{D}(S)$ is Boolean.
That is: when is an inverse $\wedge$-semigroup pre-Boolean?  
This section is in three parts corresponding to the three statements in the Booleanization Theorem given in the Introduction.

%%%%%%%%%%%%%%%%%%%%%%%%%%%%%%%%%%%%%%%%%%%%%%%%%%%%%%%%%%%%%%%%%%%%%%%%%%%%%%%%%%%%%%%%%%%%%%%%%%%%%%%%%%%%%%%%%%%%%%%%%%%%%%%%%%%%%%%%%
\begin{center}
{\bf Part 1 }
\end{center}

We show first that the problem of determining whether an inverse $\wedge$-semigroup is pre-Boolean can be answered by looking only at its semilattice of idempotents.
In what follows we shall denote the set of ultrafilters in $E(S)$ containing the element $e$ by $U_{e}$ to avoid ambiguity.

The inverse semigroup $\mathsf{D}(S)$ is isomorphic to the inverse subsemigroup $\mathsf{Bi}(\mathsf{G}(S))$ that consists of all local bisections
of the form $\bigcup_{i=1}^{m} V_{a_{i}}$ where $a_{i} \in S$.
The proof of the following is immediate.

\begin{lemma} The semilattice $E(\mathsf{D}(S))$ consists of all those bisections of the form $\bigcup_{i=1}^{m} V_{a_{i}}$ 
where each $V_{a_{i}}$ consists of only idempotent ultrafilters.
The partial ordering is subset inclusion.
\end{lemma}

The following result is a version of Proposition~2.13 \cite{Law11}.
It enables us to compare the set of ultrafilters in $S$ with the set of ultrafilters in $E(S)$. 

\begin{lemma} Let $S$ be an inverse $\wedge$-semigroup.
\begin{enumerate}

\item If $A$ is an idempotent filter in $S$ then $E(A)$ is a filter in $E(S)$ and $A = E(A)^{\uparrow}$.

\item If $F$ is a filter (respectively, an ultrafilter) in $E(S)$ then $F^{\uparrow}$ is an idempotent filter (respectively, ultrafilter) in $S$
such that $E(F^{\uparrow}) = F$.

\item $A$ is an idempotent ultrafilter in $S$ iff $E(A)$ is an ultrafilter in $E(S)$.

\item The function $\epsilon \colon \mathsf{G}(S)_{o} \rightarrow \mathsf{G}(E(S))$ defined by $\epsilon (A) = E(A)$ is a bijection.

\item Let $V_{a}$ consist entirely of idempotent ultrafilters.
Then $\epsilon \colon V_{a} \rightarrow U_{e_{a}}$ is a bijection where $e_{a} = a \wedge a^{-1}a$.

\end{enumerate}
\end{lemma}
\begin{proof}
(1) Let $A$ be an idempotent filter in $S$.
By Lemma~2.8(2), it follows that $E(A)$ is non-empty.
Since $E(A) \subseteq A$ we have that $E(A)^{\uparrow} \subseteq A$.
Let $a \in A$.
By Lemma~2.8(2), we know that idempotent filters are inverse subsemigroups.
Thus $a^{-1}a \in A$.
But $A$ is a filter and so $e = a \wedge a^{-1}a \in A$, an idempotent.
But $e \leq a$ by construction.
It follows that $a \in E(S)^{\uparrow}$ and so $A = E(A)^{\uparrow}$.

(2) The fact that $F$ is a filter in $E(S)$ implies that $F^{\uparrow}$ is a filter in $S$.
This filter contains idempotents by construction and so be Lemma~2.8(2), it is an idempotent filter.
We have that $F \subseteq F^{\uparrow}$ and so $F \subseteq E(F^{\uparrow})$.
Let $e \in   E(F^{\uparrow})$.
Then $f \leq e$ for some $f \in F$.
But $F$ is a filter in $E(S)$ and so $f \in F$.

Suppose now that $F$ is an ultrafilter in $E(S)$.
Let $F^{\uparrow} \subseteq A$ where $A$ is a filter in $S$.
Clearly, $A$ is an idempotent filter and so $A = E(A)^{\uparrow}$.
But $F \subseteq E(A)$
and
$E(A)$ is a filter and so $F = E(A)$.
It follows that $F^{\uparrow} = A$ and so $F^{\uparrow}$ is an ultrafilter in $S$.

(3) Suppose that $A$ is an idempotent ultrafilter in $S$.
Let $E(A) \subseteq F$ where $F$ is a filter in $E(S)$.
Then $A = E(A)^{\uparrow} \subseteq F^{\uparrow}$.
But $A$ is an ultrafilter and $F^{\uparrow}$ is a filter and so $A = F^{\uparrow}$.
Hence $E(A) = F$ and so $E(A)$ is an ultrafilter.
The converse follows by (2).

(4) The fact that we have a bijection follows by the results above.

(5) The proof of this is straightforward.

\end{proof}

\begin{center}
{\bf Proof of the Booleanization Theorem (1)}
\end{center}

We prove that if $S$ is an inverse $\wedge$-semigroup then $E(\mathsf{D}(S))$ is isomorphic to $\mathsf{D}(E(S))$.
\begin{proof} By Lemma~2.17 we may identity a typical element of  $E(\mathsf{D}(S))$ with a bisection of the form  $\bigcup_{i=1}^{m} V_{a_{i}}$ 
where each $V_{a_{i}}$ consists of only idempotent ultrafilters.
Define a function
$$\mathcal{E} \colon E(\mathsf{D}(S)) \rightarrow \mathsf{D}(E(S))$$
by
$$\mathcal{E}(\bigcup_{i} V_{a_{i}}  ) = \bigcup_{i} U_{e_{a_{i}}}.$$
To prove that this map is well-defined we need to prove the following.
Suppose that
$$(a_{1}, \ldots, a_{m}) \leftrightarrow (b_{1}, \ldots, b_{n})$$
and that $a_{i} \leftrightarrow e_{a_{i}}$ and $b_{j} \leftrightarrow e_{b_{j}}$.
Then
$$(e_{a_{1}}, \ldots, e_{a_{m}}) \leftrightarrow (e_{b_{1}}, \ldots, e_{b_{n}}).$$
This is straightforward.
It follows by Lemma~2.18 that this sets up an order isomorphism between $E(\mathsf{D}(S))$ and $\mathsf{D}(E(S))$.
\end{proof}

%%%%%%%%%%%%%%%%%%%%%%%%%%%%%%%%%%%%%%%%%%%%%%%%%%%%%%%%%%%%%%%%%%%%%%%%%%%%%%%%%%%%%%%%%%%%%%%%%%%%%%%%%%%%%%%%%%%%%%%%%%%%%%%%%%%%%%%%%%%%%%%%%
\begin{center}
{\bf Part 2}
\end{center}

In this section, we shall obtain necessary and sufficient conditions on an inverse $\wedge$-semigroup that it be pre-Boolean.
We showed above that we may restrict attention to semilattices so from now on $E$ will denote a meet semilattice with zero.
We denote by $\mathsf{F}(E)$ the set of all filters on $E$ and by $\mathsf{U}(E)$ the set of all ultrafilters on $E$.
Clearly $\mathsf{U}(E) \subseteq \mathsf{F}(E)$.

\begin{lemma} Let $S$ be an inverse semigroup.
\begin{enumerate}

\item Every ultrafilter is a tight filter.

\item $A$ is is a tight filter if and only if $(A^{-1}A)^{\uparrow}$ is a tight filter.

\item Let $A$ be an inverse subsemigroup. Then $A$ is a tight filter in $S$ if and only if $E(A)$ is a tight filter in $E(S)$.

\item Every tight filter in $S$ is an ultrafilter if and only if every tight filter in $E(S)$ is an ultrafilter.

\end{enumerate}
\end{lemma}
\begin{proof}
(1) Let $F$ be an ultrafilter.
Let $a \in F$ and suppose that $\{a_{1}, \ldots, a_{m} \}$ is a cover of $a$.
Suppose that $\{a_{1}, \ldots, a_{m} \} \cap F = \emptyset$.
Then by Lemma~2.6(2), we may find an element $f \in F$ such that $f \wedge a_{i} = 0$ for all $i$.
Now $f \wedge a \in F$.
By assumption there exists $a_{i}$ such that $f \wedge a \wedge a_{i} \neq 0$.
Thus $f \wedge a_{i} \neq 0$, which is a contradiction.
Therefore we must have that $a_{i} \in F$ for some $i$,
and so $F$ is a tight filter.

(2) Let $A$ be a tight filter.
We prove that $H = (A^{-1}A)^{\uparrow}$ is a tight filter.
Let $x \in H$ and let $\{x_{1}, \ldots, x_{m} \}$ be a cover of $x$.
By definition $a_{1}^{-1}a_{2} \leq x$ for some $a_{1},a_{2} \in A$.
Let $a \in A$ be arbitrary.
Then $aa_{1}^{-1}a_{2} \leq ax$.
But $aa_{1}^{-1}a_{2} \in A$ and so $ax \in A$.
We claim that $\{ax_{1}, \ldots, ax_{m} \}$ is a cover of $ax$.
Let $0 \neq y \leq ax$.
Observe that $aa^{-1}y = y$.
It follows that $a^{-1}y \neq 0$.
But $a^{-1}y \leq aa^{-1}x \leq x$.
Thus $a^{-1}y \wedge x_{i} \neq 0$ for some $i$.
But $a(a^{-1}y \wedge x_{i}) = aa^{-1}y \wedge ax_{i}$
and since $a^{-1}y \wedge x_{i} \leq a^{-1}y$ we have that
$a^{-1}a (a^{-1}y \wedge x_{i}) = a^{-1}y \wedge x_{i}$.
It follows that  $aa^{-1}y \wedge ax_{i} = y \wedge ax_{i} \neq 0$.
By assumption, $A$ is a tight filter and so $ax_{j} \in A$ for some $j$.
Thus $a^{-1}ax_{j} \in A^{-1}A$ and so $x_{j} \in H$, as required.

Suppose now that $H = (A^{-1}A)^{\uparrow}$ is a tight filter.
We prove that $A$ is a tight filter.
Let $x \in A$ and let $\{x_{1}, \ldots, x_{m} \}$ be a cover of $x$.
Let $a \in A$ be arbitrary.
Then we may prove as above that $\{a^{-1}x_{1}, \ldots, a^{-1}x_{m} \}$ is a cover of $a^{-1}x$.
It follows that $a^{-1}x_{j} \in H$ for some $j$.
Thus $x_{j} \in A$, as required.

(3) It is immediate that $E(A)$ is a tight filter in $E(S)$ if $A$ be a tight filter and inverse subsemigroup in $S$.
We prove the converse.
Let $a \in A$ and let $\{a_{1}, \ldots, a_{m} \}$ be a cover of $a$.
Then $\dom (a) \in A$, since $A$ is an inverse subsemigroup and it is easy to check that
$\{\dom (a_{1}), \ldots, \dom (a_{m}) \}$ is a cover of $\dom (a)$.
By assumption, $\dom (a_{i}) \in E(A)$ for some $i$.
But $a,\dom (a_{i}) \in A$ implies that $a_{i} = a \dom (a_{i}) \in A$ since $A$ is an inverse subsemigroup.
We have therefore proved that $A$ is a tight filter.

(4) Suppose that every prime filter in $S$ is an ultrafilter.
Let $F$ be a tight filter in $E(S)$.
Then $F^{\uparrow}$ is a tight filter in $S$ by (3) above.
By assumption $F^{\uparrow}$ is an ultrafilter and so by Lemma~2.18, $F$ is an ultrafilter in $E(S)$.

Suppose now that every tight filter in $E(S)$ is an ultrafilter.
Let $A$ be a tight filter in $S$.
Then $H = (A^{-1}A)^{\uparrow}$ is a prime filter and inverse subsemigroup of $S$ by (2) above.
Thus by (3) above $E(H)$ is a tight filter in $E(S)$.
It follows that $E(H)$ is an ultrafilter in $E(S)$.
Thus by Lemma~2.18, $H$ is an ultrafilter in $S$.
It follows by Proposition~2.13 of \cite{Law11} that $A$ is an ultrafilter in $S$.
\end{proof}

We denote the set of tight filters in $E$ by $\mathsf{T}(E)$.
We therefore have that
$$\mathsf{U}(E) \subseteq \mathsf{T}(E) \subseteq \mathsf{F}(E).$$

\begin{remark}{\em  In Definition~2.6 of \cite{E}, Exel defines a tight representation to be a homomorphism $\beta \colon E \rightarrow B$
to a {\em unital} boolean algebra $B$ satisfying certain conditions.
From Definition~4.5 of \cite{E}, a filter $F$ is defined to be tight if its characteristic function $\chi_{F} \colon E \rightarrow \mathbf{2}$,
to the two-element boolean algebra, is tight.
For such homomorphisms, condition (i) of Proposition~2.7 of \cite{E} holds.
It is now easy to check from this that $F$ is tight in the sense of Exel \cite{E} if and only if it is tight in the sense of this paper.
}
\end{remark}

The following property gives the essence of tight filters.

\begin{lemma} Let $X = \{x_{1}, \ldots, x_{m}\}$ and $Y = \{y_{1}, \ldots, y_{n} \}$ 
be finite subsets of $E$ with $X$ non-empty 
and let $F$ be a tight filter such that $X \subseteq F$ and $Y \cap F = \emptyset$.
Then
\begin{enumerate}

\item $\comp \neq 0$.

\item If $Z$ is any finite cover of $\comp$ then $Z \cap F \neq \emptyset$.

\end{enumerate}
\end{lemma}
\begin{proof} 
(1) Suppose that  $\comp = \{ 0\}$.
Then $\emptyset$ is a finite cover.
Put $x = \bigwedge_{i} x_{i} \in F$.
We claim that $\{x \wedge y_{1}, \ldots, x \wedge y_{n} \}$ is a cover of $x$.
Let $0 \neq y \leq x$.
Then by assumption $y \wedge y_{j} \neq 0$ for some $j$.
But $y = y \wedge x$ which proves the claim.
However, $F$ is a tight filter and so $x \wedge y_{k} \in F$ for some $k$.
But this implies that $y_{k} \in F$, which is a contradiction.

(2) From (1) above, any cover $Z$ is non-empty.
Put $x = \bigwedge_{i} x_{i} \in F$.
Let $\{z_{1}, \ldots, z_{p}\}$ be a cover of $\{x\}^{\downarrow} \cap \{y_{1}, \ldots, y_{n} \}^{\perp}$.
We claim that
$Z' = \{z_{1}, \ldots, z_{p}, x \wedge y_{1}, \ldots, x \wedge y_{n}\}$ is a cover of $x$.
Let $0 \neq y \leq x$.
Suppose that $y$ is orthogonal to all the $y_{i}$.
Then $y \wedge z_{k} \neq 0$ for some $k$.
It follows that $Z' \cap F \neq \emptyset$.
But clearly this intersection cannot contain any of the $x \wedge y_{j}$ and so must in fact be a non-empty intersection with $Z$.
\end{proof}

For each $e \in E$, define
$$U_{e} = \{F \in \mathsf{F}(E) \colon e \in F \}.$$
Observe that $U_{0} = \emptyset$.
Given $e$ and a finite set $\{e_{1}, \ldots, e_{n}\}$, possibly empty, 
define
$$U_{e:e_{1}, \ldots, e_{n}} = U_{e} \cap U_{e_{1}}^{c}   \cap \ldots \cap U_{e_{n}}^{c},$$
the set of all filters that contain $e$ and omit all of $e_{1}, \ldots, e_{n}$,
where we write $X^{c}$ to denote the set-theoretic complement of the set $X$.
Observe that
$$U_{e:e_{1}, \ldots, e_{n}}
=
U_{e:e_{1} \wedge e, \ldots, e_{n} \wedge e}$$
so that we may assume $e_{1}, \ldots, e_{n} \leq e$ if it simplifies calculations.

\begin{lemma} 
The sets above form the basis for a topology on $\mathsf{F}(E)$ which is Hausdorff.
\end{lemma}
\begin{proof} Let 
$$A \in U_{e:e_{1}, \ldots, e_{m}} \cap  U_{f:f_{1}, \ldots, f_{n}}.$$ 
Consider $U = U_{e \wedge f:e_{1}, \ldots, e_{m},f_{1}, \ldots, f_{n}}$.
Clearly $A \in U$ and 
$$U \subseteq  U_{e:e_{1}, \ldots, e_{m}} \cap  U_{f:f_{1}, \ldots, f_{n}}.$$ 
It follows that the sets do form a basis.

Let $A,B \in \mathsf{F}(E)$ where $A \neq B$.
If $A \setminus B$ and $B \setminus A$, with the set-theoretic meanings,
are both non-empty, then choose $e \in A \setminus B$ and $f \in B \setminus A$.
We have that $A \in U_{e: e \wedge f}$ and $B \in U_{f : e \wedge f}$ and
$U_{e: e \wedge f} \cap U_{f : e \wedge f} \neq \emptyset$.
If, say, $A$ is properly contained in $B$ then choose $e \in A$ and $f \in A \setminus B$.
Observe that $B \in U_{f}$ and $A \in U_{e: e \wedge f}$ and
$U_{f}$ and $A \cap U_{e: e \wedge f} = \emptyset$.
Thus the topology is Hausdorff. \end{proof}

The following was proved in \cite{Lenz} as Proposition~4.7 using functional analysis. 
We give a direct elementary proof.

\begin{lemma} 
The sets $U_{e:e_{1}, \ldots, e_{n}}$ are compact-open
\end{lemma}
\begin{proof} We begin with a general construction.
Put $\mathbf{2} = \{0,1 \}$, the unital Boolean algebra with two elements.
Let $\mathbf{2}^{E}$ be the product space.
Each element is a function from $E$ to $\{0,1\}$.
By Tychonoff's theorem it is a compact space.
A subbase for this topology is given by subsets of the form $T_{e}$ and $T_{e}^{c}$,
where
$$T_{e} = \{\theta \colon E \rightarrow \mathbf{2} \colon \theta (e) = 1 \}.$$
It follows that these sets are clopen
and so the sets
$T_{e} \cap T_{e_{1}}^{c} \cap \ldots \cap T_{e_{n}}^{c}$
are compact-open.

Next we look at the subspace of  $\mathbf{2}^{E}$ that consists of all {\em semigroup homomorphisms} $\theta \colon E \rightarrow \{0,1 \}$.
These are functions $\theta \colon E \rightarrow \mathbf{2}$ such that
$\theta (e \wedge f) = \theta (e) \wedge \theta (f)$.
A function $\phi$ fails to be a homomorphism if and only if 
$\theta (e \wedge f) \neq \theta (e)\theta (f)$ for some $e,f \in E$
if and only if $\theta$ belongs to the union of sets of the form
$$U_{e} \cap U_{f} \cap U_{e \wedge f}^{c}
\mbox{ or }
U_{e}^{c} \cap U_{f} \cap U_{e \wedge f}
\mbox{ or }
U_{e} \cap U_{f}^{c} \cap U_{e \wedge f}
\mbox{ or }
U_{e}^{c} \cap U_{f}^{c} \cap U_{e \wedge f}
$$
for some $e,f \in E$.
It follows that the set of functions which are not semigroup homomorphisms is open
and so the set of homomorphisms is closed.
Thus the subspace $\bar{E}$ of  $\mathbf{2}^{E}$ that consists of all homomorphisms $f \colon E \rightarrow \{0,1 \}$
is the closed subset of all semigroup homomorphisms intersected with the closed set $T_{0}$ which forces $\phi (0) = 0$.
This is therefore a closed subset.

Finally, we remove from this space the homomorphism that sends everything to zero.
We therefore obtain the locally compact space $\hat{E}$ of all non-zero homomorphisms from $E$ to $\{0,1 \}$.

There is a bijection between $\mathsf{F}(E)$ and $\hat{E}$ 
which takes a filter $F$ to its characteristic function $\chi_{F}$,
and which associates with each non-zero homomorphism the set of all elements of $E$ that map to 1.
This bijection is actually a homeomorphism.

We now prove the lemma.
We may assume that $e \neq 0$ since no filter contains 0.
This set is mapped by our bijection above to the set
$T_{e} \cap T_{e_{1}}^{c} \cap \ldots \cap T_{e_{n}}^{c}$
intersected with the closed subset $\bar{E}$ of homomorphisms.
It is therefore a closed subset of $\mathbf{2}^{E}$ 
and so is itself compact.
\end{proof}

We summarize what we have found in the following.

\begin{proposition} 
The topological space $\mathsf{F}(E)$ of all filters on $E$ is Hausdorff with a basis of compact-open subsets.
\end{proposition}

The following is the analogue of Theorem~4.4 of \cite{E}.

\begin{proposition} The closure of the set $\mathsf{U}(E)$ of ultrafilters in the space of all filters $\mathsf{F}(E)$
is precisely the set of tight filters $\mathsf{P}(E)$.
\end{proposition}
\begin{proof}
We show first that $\mathsf{T}(E)$ is a closed subset of $\mathsf{F}(E)$.
Let $F$ be a filter which is not tight.
Then by definition there is an element $f \in F$ and a cover $\{e_{1}, \ldots, e_{m} \}$ of $f$ such that $F \cap \{e_{1}, \ldots, e_{m} \} = \emptyset$.
Thus $F \in U_{f:e_{1}, \ldots, e_{m}}$ an open set.
Let $G \in U_{f:e_{1}, \ldots, e_{m}}$ be an ultrafilter.
Then since $\{e_{1}, \ldots, e_{m}\}$ is a cover of $f$ we must have that $e_{i} \in G$ for some $i$ which is a contradiction.

Let $F$ be a tight filter on $E$.
Let $F \in \mathcal{U}$ be any open subset of $\mathsf{F}(E)$.
Then there is an open subset $F \in \mathcal{V} \subseteq \mathcal{U}$
such that $\mathcal{V}$ contains an ultrafilter.
To see why, we first observe that from the definition of the topology, we may find an open set $\mathcal{V}$
of the form $U_{e: e_{1}, \ldots, e_{m}}$ containing $F$.
Thus $e \in F$ and $\{e_{1}, \ldots, e_{m}\} \cap F = \emptyset$.
Since $F$ is a tight filter, there is a non-zero $z \in e^{\downarrow} \cap \{e_{1}, \ldots, e_{m} \}$ by Lemma~2.21.
Let $G$ be any ultrafilter containing $z$; such exists by Lemma~2.6(1).
Observe that $z \leq e$ and so $e \in G$, and that $z \wedge e_{i} = 0$ and so $e_{i} \notin G$.
It follows that $G \in U_{e: e_{1}, \ldots, e_{m}}$.
\end{proof}

Define 
$$V_{e:e_{1}, \ldots, e_{n}} = U_{e:e_{1}, \ldots, e_{n}} \cap \mathsf{U}(E).$$
This defines the subspace topology on $\mathsf{U}(E)$.

\begin{lemma} 
The above topology has $\{V_{e} \colon e \in E \}$ as a basis.
\end{lemma}
\begin{proof}
Let $F \in V_{e:e_{1}, \ldots, e_{n}}$ be an ultrafilter.
Then $e \in F$ and $\{e_{1}, \ldots, e_{n} \} \cap F = \emptyset$.
Thus by Lemma~2.21, there exists $f \in F \cap e^{\downarrow} \cap \{e_{1}, \ldots, e_{n} \}^{\perp}$.
Clearly $F \in V_{f}$.
Let $G \in V_{f}$ be any ultrafilter.
Then $f \in G$ and $f \leq e$ implies that $e \in G$ 
and if $e_{i} \in G$ then $f \wedge e_{i} = 0$.
Thus in fact $G \in V_{e:e_{1}, \ldots, e_{n}}$.
Hence $F \in V_{f} \subseteq  V_{e:e_{1}, \ldots, e_{n}}$.
Finally, observe that $V_{0} = \emptyset$ and that $V_{e \wedge f} = V_{e} \cap V_{f}$.
\end{proof}

The following result could be proved by using Proposition~6.3 of \cite{Lenz} and Proposition~2.24,
but given its importance, we prefer to give a direct proof.

\begin{proposition} 
The sets $V_{e:e_{1}, \ldots, e_{n}}$ are all compact in $\mathsf{F}(E)$ if and only if every tight filter is an ultrafilter.
\end{proposition} 
\begin{proof}
Suppose that every tight filter is an ultrafilter.
Then $\mathsf{U}(E)$ is a closed subset of $\mathsf{F}(E)$ by Proposition~2.25.
By Proposition~2.24, $\mathsf{F}(E)$ is a Hausdorff space.
By Lemma~2.23, the sets $U_{e: e_{1}, \ldots, e_{n}}$ are compact-open.
Therefore they are closed.
Thus the set $V_{e:e_{1}, \ldots, e_{n}} \subseteq U_{e:e_{1}, \ldots, e_{n}}$ is closed.
But a closed subset of a compact space is itself compact.
Thus the subsets $V_{e:e_{1}, \ldots, e_{n}}$ are compact.

Suppose now that all sets $V_{e:e_{1}, \ldots, e_{n}}$ are compact in $\mathsf{F}(U)$.
We prove that every tight filter is an ultrafilter.
Let $F$ be a tight filter in $\mathsf{F}(E)$.
Let $e \in F$ so that $F \in U_{e}$.
Let $O$ be any open set containing $F$.
Then $O \cap U_{e}$ is an open set containing $F$.
Since $F$ is tight we have that $O \cap U_{e}$ contains an ultrafilter by Proposition~2.25.
It follows that $O \cap V_{e}$ is non-empty.
The set $V_{e}$ is compact in the Hausdorff space $\mathsf{F}(U)$ and so is closed.
However, we have shown that $F$ is a limit point for $V_{e}$ and so must belong to $V_{e}$
and is therefore an ultrafilter, as required.
\end{proof}

%%%%%%%%%%%%%%%%%%%%%%%%%%%%%%%%%%%%%%%%%%%%%%%%%%%%%%%%%%%%%%%%%%%%%%%%%%%%%%%%%%%%%%%%%%%%%%%%%%%%%%%%%%%%%%%%%%%%%%%%%%%%%%%%%%%%%%%%%%%%%%%%%%%%%%
\begin{center}
{\bf Proof of the Booleanization Theorem (2)}
\end{center}

Recall that the elements of $\mathsf{D}(E)$ are the sets of the form $\bigcup_{i=1}^{m} V_{e_{i}}$ and that the partial ordering is subset inclusion.
We shall prove that $\mathsf{D}(E)$ is a Boolean algebra if and only if every tight filter in $E$ is an ultrafilter.
\begin{proof}
Suppose first that every tight filter is an ultrafilter.
Let $A \subseteq B$ where 
$A = \bigcup_{i=1}^{m} V_{e_{i}}$
and
$B = \bigcup_{j=1}^{n} V_{f_{i}}$.
The set $B \slash A$ is a finite union of sets of the form $V_{e} \slash V_{f}$.
Observe that 
$V_{e} \slash V_{f} = V_{e: e \wedge f}$.
Under our assumption, $V_{e: e \wedge f}$ is compact and so by Lemma~2.26,
it is a finite union of sets of the form $V_{g}$.
It follows that $B \slash A$ is a finite union of sets of the form $V_{g}$.
Hence $B \slash A \in \mathsf{D}(E)$ and so, since $\mathsf{D}(E)$ is already a distributive lattice, it is a Boolean algebra.

To prove the converse, assume that $\mathsf{D}(E)$ is a Boolean algebra.
In our proof below, we use the fact that in a Boolean algebra, not necessarily unital,
every ultrafilters are the same as prime filters and that maps to $\mathbf{2}$ are determined by ultrafilters.
We now prove that every tight filter is an ultrafilter.
Let $F$ be a tight filter in $E$.
Then the characteristic function $\chi_{F} \colon E \rightarrow \mathbf{2}$ is a cover-to-join map.
Thus by the universal property of the map $\delta \colon E \rightarrow \mathsf{D}(E)$ there is a unique
homomorphism $\chi_{G} \colon \mathsf{D}(E) \rightarrow \mathbf{2}$ such that $\chi_{G} \delta = \chi_{F}$.
The filter $G$ is an ultrafilter because $\mathsf{D}(E)$ is a boolean algebra.
Since the diagram of maps commutes we also have that $F = \delta^{-1}(G)$.
We shall prove that $F$ is an ultrafilter.
Suppose not.
Then by Lemma~2.6(2), there is an $e \in E$ such that $e$ has a non-zero intersection with every element of $F$ but $e \notin F$.
It follows that $\delta (e) \notin G$.
But $\delta$ is $0$-restricted and so $\delta (e) \neq 0$.
Now $G$ is an ultrafilter and so there exists $g \in G$ such that
$\delta (e) \wedge g = 0$.
Each non-zero element of $\mathsf{D}(S)$ is a join of a finite number of elements in the image of $\delta$.
Thus $g = \bigvee_{i=1}^{m} \delta (e_{i})$.
But $G$ is an ultrafilter and so a prime filter. 
It follows that $\delta (e_{i}) \in G$ for some $i$.
Also $\delta (e) \wedge \delta (e_{i}) = 0$.
Again using the fact that $\delta$ is $0$-restricted, we have that $e \wedge e_{i} = 0$.
But $e_{i} \in F$.
This is a contradiction.
It follows that $F$ is a tight filter.
\end{proof}

One obvious question is when the Boolean completion of a pre-Boolean semilattice is unital.
A finite set of idempotents  $\{e_{1}, \ldots, e_{n} \}$ in a semilattice $E$ is said to be {\em essential}
if it is a cover of $E \setminus\{0\}$.
We shall say that a pre-Boolean semilattice $E$ is {\em compactable} if it has at least one essential set of idempotents.

\begin{theorem} Let $E$ be a pre-Boolean semilattice.
Then $\mathsf{D}(E)$ is unital if and only if $E$ is compactable,
and the essential sets of idempotents are precisely the sets of idempotents mapped to the identity of  $\mathsf{D}(E)$.
\end{theorem}
\begin{proof} Suppose that $\mathsf{D}(E)$ is unital.
Then $\mathsf{D}(E) = \bigcup_{i=1}^{n} V_{e_{i}}$ for some finite set of idempotents $e_{1}, \ldots, e_{n}$.
Let $e \in E$ be an arbitrary non-zero element.
Then by Lemma~2.6(1), there is an ultrafilter $F$ containing $e$.
By assumption, $e_{i} \in F$ for some $i$.
In particular, $e \wedge e_{i} \neq 0$.
Thus $\{e_{1}, \ldots, e_{n} \}$ is a finite cover.

Conversely, suppose that $\{e_{1}, \ldots, e_{n} \}$ is an essential set of idempotents.
Let $F$ be any ultrafilter.
Suppose that $e_{i} \notin F$ for all $i$.
Then for each $i$ there exists $f_{i} \in F$ such that $f_{i} \wedge e_{i} = 0$ by Lemma~2.6(2).
Put $f = \wedge_{i=1}^{n} f_{i}$.
Then $f \in F$ and so is non-zero.
But $f \wedge e_{i} = 0$ for all $i$.
This contradicts our assumption that $\{e_{1}, \ldots, e_{n} \}$ is a finite cover.
Thus we must have $e_{i} \in F$ for some $i$.
It follows that 
$\mathsf{D}(E) = \bigcup_{i=1}^{n} V_{e}$ and so $\mathsf{D}(E)$ is unital.
\end{proof}

%%%%%%%%%%%%%%%%%%%%%%%%%%%%%%%%%%%%%%%%%%%%%%%%%%%%%%%%%%%%%%%%%%%%%%%%%%%%%%%%%%%%%%%%%%%%%%%%%%%%%%%%%%%%%%%%%%%%%%%%%%%%%%%%%%%%%%%%%%%%%%%%%%%
\begin{center}
{\bf Part 3 }
\end{center}

\begin{lemma} Let $X$ and $Y$ be finite subsets of $E$ with $X$ non-empty.
Put $e = \bigwedge X$ and let $Y = \{e_{1}, \ldots, e_{n} \}$.
Assume that $\comp$ is non-zero.
Then $Z \subseteq X^{\wedge} \cap Y^{\perp}$ is a finite cover if and only if 
$$\bigcup_{z \in Z} V_{z} = V_{e: e_{1}, \ldots, e_{n}}.$$
\end{lemma}
\begin{proof}
Suppose that
$$\bigcup_{z \in Z} V_{z} = V_{e: e_{1}, \ldots, e_{n}}.$$
We prove that  $Z \subseteq X \cap Y^{\perp}$ is a cover.
Let $x \in  X^{\wedge} \cap Y^{\perp}$.
Then $x \leq e$ and $x \wedge e_{i} = 0$ for all $i$.
Let $F$ be an ultrafilter containing $x$.
Then $F$ contains $e$ and omits all the $e_{i}$.
Thus $F \in V_{e:e_{1}, \ldots, e_{n}}$.
By assumption $z \in F$ for some $z \in Z$.
But then $z \wedge x \neq 0$, as required.

Conversely, suppose that $Z \subseteq X^{\wedge} \cap Y^{\perp}$ is a finite cover. 
Let $F \in V_{e: e_{1}, \ldots, e_{n}}$.
Then $F$ contains $e$ and omits all the $e_{i}$.
Let $f \in F$ such that $f \wedge e_{i} = 0$ for all $i$.
Suppose that $Z \cap F = \emptyset$.
Let $f' \in F$ be such that $z \wedge f' = 0$ for all $z \in Z$.
Then $0 \neq e \wedge f \wedge f' \in X^{\wedge} \cap Y^{\perp}$.
Thus there is $z \in Z$ such that $z \wedge e \wedge f \neq 0$
which is a contradiction.
Thus $Z \cap F \neq \emptyset$.
\end{proof}

Let $E$ be a $0$-disjunctive semilattice.
We now prove that $E$ is pre-Boolean if and only if it satisfies the trapping condition.

\begin{center}
{\bf Proof of the Booleanization Theorem (3)}
\end{center}

\begin{proof} Suppose that $E$ is $0$-disjunctive and pre-Boolean.
Let $0 \neq f < e$.
Because $E$ is $0$-disjunctive, the set $e^{\downarrow} \cap f^{\perp} \neq \{ 0\}$. 
The set $V_{e:f}$ is non-empty and by assumption compact and so by Lemma~2.29,
$e^{\downarrow} \cap f^{\perp}$ has a finite cover.

Suppose now that the trapping condition holds.
We prove that every tight filter is an ultrafilter.
Let $F$ be a tight filter that is not an ultrafilter.
Then we can find an ultrafilter $G$ such that $F \subset G$.
Let $g \in G \setminus F$ and let $f \in F$ be arbitrary.
Then $g' = g \wedge f \in G$.
We shall prove that $F$ tight implies that $g' \in F$ which implies $g \in F$, a contradiction.
We have that $0 \neq g' < f$.
Since $E$ is $0$-disjunctive, the set $f^{\downarrow} \cap (g')^{\perp} \neq 0$.
Let $\{f_{1}, \ldots, f_{n} \}$ be a cover.
Suppose that $f_{i} \in F$.
Then $f_{i},g' \in G$ and so $f_{i} \wedge g' \neq 0$.
But this contradicts the fact that $f_{i} \wedge g' = 0$.
It follows that $f \in F$ and $\{f_{1}, \ldots, f_{n} \} \cap F = \emptyset$.
Observe that $g' \in f^{\downarrow} \cap \{f_{1}, \ldots, f_{n}\}^{\perp}$.
Let $i$ be any nonzero element of  $f^{\downarrow} \cap \{f_{1}, \ldots, f_{n}\}^{\perp}$.
If $i \wedge g' = 0$ then $i \wedge f_{i} \neq 0$ for some $i$ which is a contradiction.
It follows that $i \wedge g' \neq 0$.
Thus $\{ g' \}$ is a cover of  $f^{\downarrow} \cap \{f_{1}, \ldots, f_{n}\}^{\perp}$.
By Lemma~2.21, it follows that $g' \in F$ giving our contradiction.
It follows that $F$ is an ultrafilter.
\end{proof} 

The semilattices arising above can be characterized in a more direct way.
We say that a semilattice $(E,\wedge)$ is {\em densely embedded} in a Boolean algebra $B$ if 
the $\wedge$-operation on $E$ is precisely the restriction of the $\wedge$-operation in $B$ and if every non-zero element of $B$ is a finite join of elements of $E$.

\begin{proposition} 
A semilattice  is $0$-disjunctive and satisfies the trapping condition
if and only if 
it can be densely embedded in a Boolean algebra.
\end{proposition}
\begin{proof} 
Suppose that our semilattice $E$ can be densely embedded in Boolean algebra $B$.
We prove that it is $0$-disjunctive and that the trapping condition holds.
Let $0 \neq f < e$ in $E$.
Then $e^{\downarrow} \cap f^{\perp} = (e \bsl f)^{\downarrow}$.
By denseness, we may write $e \bsl f = \bigvee_{i=1}^{n} e_{i}$ where $e_{i} \in E$.
In particular, this shows that $E$ is $0$-disjunctive.
Put $Z = \{e_{1}, \ldots, e_{n} \}$.
It is easy to check that this is a cover of  $e^{\downarrow} \cap f^{\perp}$.
For the converse, we apply part (3) of the Booleanization Theorem and observe by Lemma~2.2 
that $0$-disjunctive semilattices are separative.
\end{proof}

%%%%%%%%%%%%%%%%%%%%%%%%%%%%%%%%%%%%%%%%%%%%%%%%%%%%%%%%%%%%%%%%%%%%%%%%%%%%%%%%%%%%%%%%%%%%%%%%%%%%%%%%%%%%%%%%%%%%%%%%%%%%%%%%%%%%%%%%%%%%%
\subsection{The Comparison Theorem}

The main theorem this section shows how to construct a Hausdorff Boolean groupoid directly from a {\em pre-Boolean} inverse $\wedge$-semigroup.

With each pre-Boolean inverse semigroup $S$ we may associate a Boolean inverse semigroup $\mathsf{D}(S)$.
Our proof of this in Section~2.1 used the groupoid of ultrafilters of $S$.
We show now how to regard this as a Hausdorff Boolean groupoid $\mathsf{G}(S)$ 
and establish the exact connection between it and the Boolean inverse $\wedge$-semigroup $\mathsf{D}(S)$.

Let $G$ be a topological groupoid.
We make the following definitions:
$\Omega (G)$ is the set of all open subsets of $G$;
$\mathsf{Bi}(G)$ is the the inverse semigroup of all local bisections of $G$;
$\mathsf{oBi}(G) = \mathsf{Bi}(G) \cap \Omega (G)$ is the set of all open local bisections of $G$;
$\mathsf{B}(G)$ is the set of all compact-open local bisections of $G$.
The proofs of the following can be found in \cite{Resende}.

\begin{lemma} Let $G$ be a topological groupoid.
\begin{enumerate}

\item $G$ is open if and only if the multiplication map $\mathbf{m}$ is open if and only if $\Omega (G)$ is a semigroup under the pointwise product.

\item If $G$ is open then $\mathsf{oBi}(G)$ is an inverse semigroup. 

\item $G$ is \'etale if and only if it is open and $G_{o}$ is open in $G$.

\item If $G$ is \'etale then $\Omega (G)$ is a monoid and $\mathsf{oBi}(G)$ is an inverse monoid.

\item If $G_{o}$ is open in $G$ then the open bisections form a basis for the topology on $G$.
 
\item If $G$ is \'etale then $\mathsf{oBi}(G)$ is an inverse monoid and a basis for the topology $\Omega (G)$.

\end{enumerate}
\end{lemma}

\begin{lemma} Let $G$ be a Hausdorff, \'etale topological groupoid.
\begin{enumerate}

\item $G \ast G$ is a closed subspace of $G \times G$.

\item The product of two compact-open bisections is a compact-open bisection.

\item $G$ has a basis of compact-open bisections if and only if $G_{o}$ has a compact-open basis.

\item If $G$ is a Hausdorff Boolean groupoid then $\mathsf{B}(G)$ is a Boolean inverse $\wedge$-semigroup.
 
\end{enumerate}
\end{lemma}
\begin{proof} (1) This is well-known but we give the proof for completeness' sake.
Put $\Delta (G_{o}) = \{ (e,e) \colon e \in G_{o} \}$.
Now $G$ is hausdorff and $G_{o}$ is an open subset of $G$ by Lemma~2.36.
Thus $G_{o}$ is hausdorff.
By Theorem~13.7 \cite{W}, it follows that $\Delta (G_{o})$ is a closed subspace of $G_{o} \times G_{o}$.
But $G \ast G = (\mathbf{d}, \mathbf{r})^{-1}(\Delta (G_{o}))$ and the map $(\mathbf{d}, \mathbf{r})$ is continuous.
Thus $G \ast G$ is a closed subsapce as claimed.

(2) Let $A$ and $B$ be compact-open bisections.
It only remains to prove that $AB$ is compact.
By Theorem~D of Section~26 of \cite{S}, both $A$ and $B$ are closed subsets of $G$.
It follows that $A \times B$ is a closed subset of $G \times G$.
By (1), we therefore have that $A \ast B = (a \times B) \cap (G \ast G)$ is a closed subset.
But $A \times B$ is compact and so by Theorem~A of Section~21 of \cite{S} the set $A \ast B$ is compact.
By Theorem~B of Section~21 of \cite{S}, the continuous image of a compact space is compact.
Thus $AB$ is compact.

(3) Suppose that $G$ has a basis of compact-open bisections.
We prove that $G_{o}$ has a basis of compact-open sets.
Consider the set of all compact-open bisections that are subsets of $G_{o}$.
We prove that they form a basis for the subspace topology on $G_{o}$.
Every open set $U$ of $G_{o}$ has the form $U = G_{o} \cap V$ where $V$ is an open set in $G$.
But $G_{o}$ is an open subspace of $G$ and so $U$ is also an open subset of $G$.
By assumption, we may write $U = \bigcup_{i} B_{i}$ where $B_{i}$ are compact-open bisections.
But $B_{i} \subseteq U \subseteq G_{o}$ and so each $B_{i}$ is a compact-open bisection contained in $G_{o}$.
It is clear that the $B_{i}$ are open in $G_{o}$ and also compact.

Suppose that $G_{o}$ has a basis of compact-open sets.
We prove that $G$ has a basis of compact-open bisections.
By assumption, $G$ has a basis of open bisections.
Let $A$ be such an open bisection.
Then $A^{-1}A$ is an open subset of $G_{o}$.
By assumption we may write $A^{-1}A = \bigcup_{i}E_{i}$
where the $E_{i}$ are compact-open subsets of $G_{o}$.
Put $A_{i} = AE_{i}$.
This is an open bisection satisfying $A_{i}^{-1}A_{i} = E_{i}$.
We shall have proved the result if we can prove that 
every open bisection $A$ such that $A^{-1}A$ is compact-open is itself compact-open.
Suppose that $A = \bigcup_{j} B_{j}$ where the $B_{j}$ are open bisections.
Then $A^{-1}A =  \bigcup_{j} B_{j}^{-1}B_{j}$.
By assumption $A^{-1}A$ is compact-open and so we may find a finite subset of the $j$
such that $A^{-1}A = \bigcup_{j=1}^{n} B_{j}^{-1}B_{j}$.
But then $A = \bigcup_{j=1}^{n} AB_{j}^{-1}B_{j} = \bigcup_{j=1}^{n} B_{j}$.
Thus if $A$ is an open bisection such that $A^{-1}A$ is compact-open then $A$ is compact-open.

(4) It only remains to observe that if $A$ and $B$ are both compact-open bisections then $A \cap B$ being
closed and a subset of $A$ is also compact and so compact-open and a bisection,
and if $A$ is compatible with $B$ then $A \cup B$ is a bisection, open and compact. 
\end{proof}

%%%%%%%%%%%%%%%%%%%%%%%%%%%%%%%%%%%%%%%%%%%%%%%%%%%%%%%%%%%%%%%%%%%%%%%%%%%%%%%%%%%%%%%%%%%%%%%%%%%%%%%%%%%%%%%%%%%%%%%%%%%%%%%%%%%%%%%%%%%%%
Recall that if $S$ be an inverse $\wedge$-semigroup with zero, then we defined $\Omega = \{V_{a} \colon a \in S \}$.
By Lemma~2.10(1), the set $\Omega$ is a basis for a topology on the groupoid $\mathsf{G}(S)$.

\begin{proposition} Let $S$ be an inverse $\wedge$-semigroup with zero.
Then with the topology above $\mathsf{G}(S)$ is a Hausdorff \'etale topological groupoid whose topological space of identities is 
homeomorphic to the topological space constructed from $E(S)$.
\end{proposition}
\begin{proof} We show first that the topology is Hausdorff.
Let $A,B \in \mathsf{G}(S)$ be distinct elements.
Then there exists $a \in A$ such that $a \notin B$
otherwise we would have $A \subseteq B$ and so $A = B$ since both $A$ and $B$ are ultrafilters.
But $b \notin B$ implies that there exists $b \in B$ such that $a \wedge b = 0$ by Lemma~2.6(2).
Thus $V_{a} \cap V_{b} = \emptyset$ and $A \in V_{a}$ and $B \in V_{b}$.

That inversion is continuous follows by Lemma~2.10(3).

Multiplication is continuous.
Let $\mathsf{G}(S) \ast \mathsf{G}(S)$ denoted the subset of  $\mathsf{G}(S) \times \mathsf{G}(S)$ 
consisting of pairs $(A,B)$ such that $A^{-1} \cdot A = B \cdot B^{-1}$.
We prove that the map 
$$\mu \colon \mathsf{G}(S) \ast \mathsf{G}(S) \rightarrow \mathsf{G}(S)$$ 
given by $(A,B) \mapsto A \cdot B$ is continuous.
To do this we prove that
$$\mu^{-1}(V_{a}) = \left(\bigcup_{0 \neq bc \leq a} V_{b} \times V_{c}\right) \cap (\mathsf{G}(S) \ast \mathsf{G}(S)).$$
Let $(B,C) \in \mu^{-1}(V_{a})$.
Then $A = B \cdot C$ is an ultrafilter containing $a$. 
Then $a \in (BC)^{\uparrow}$ and so $bc \leq a$ for some $b \in B$ and $c \in C$.
Thus $B \in V_{b}$, $C \in V_{c}$ and $0 \neq bc \leq a$.
We have proved that the lefthand side is contained in the righthand side.
To prove the reverse inclusion, 
suppose that $0 \neq bc \leq a$ and $B \in V_{b}$, $C \in \mathcal{K}_{c}$ and the product $B \cdot C$ exists. 
Then $B \cdot C$ is an ultrafilter containing $a$ and so $B \cdot C \in \mathcal{K}_{a}$.

The map $\mathbf{d}$ is a local homeomorphism.
We are required to show that $\dom \colon \mathsf{G}(S) \rightarrow \mathsf{G}(S)_{o}$ is a local homeomorphism.
To do this it is enough to prove that the map $\dom \colon V_{s} \rightarrow V_{s^{-1}s}$
given by $A \mapsto A^{-1} \cdot A$ is a homeomorphism. 
It is bijective by Lemma~2.8(3).
It is continuous because inversion and multiplication are continuous. 
It remains to show that it is open.
A base of open subsets of $V_{s}$ is given by the sets $V_{s} \cap V_{t}$;
that is, the sets $V_{s \wedge t}$.
It follows that a base of open subsets of  $V_{s}$ is given by the sets
$V_{t}$ where $t \leq s$.
But it follows that $V_{t^{-1}t}$ is an open set in $V_{s^{-1}s}$.

$\mathsf{G}(S)_{o}$ is homeomorphic to the space constructed from $E(S)$.
In Lemma~2.18(4), we defined a bijection $\epsilon \colon \mathsf{G}(S)_{o} \rightarrow \mathsf{G}(E(S))$ by $\epsilon (A) = E(A)$.
We prove that this is a homeomorphism.
We need some notation.
If $e \in E(S)$ is an idempotent we denote the set of ultrafilters in $E(S)$ containing $e$ by $V_{e}^{E}$.
We prove that the map $\epsilon$ is continuous and open.
By definition 
$$\epsilon^{-1} (V_{e}^{E}) = \{A \in \mathsf{G}(S)_{o} \colon e \in E(A) \}.$$
But the set on the righthand side is just $V_{e}$ because a filter contains an idempotent if and only if it 
is an inverse submonoid if and only if it is an idempotent filter.
It follows that $\epsilon$ is continuous.
To show that $\epsilon$ is open we calculate $\epsilon (V_{a} \cap \mathsf{G}(S)_{o})$.
The elements of $V_{a} \cap \mathsf{G}(S)_{o}$ are the idempotent ultrafilters that contain $a$.
They therefore contain $a^{-1}a$ and so also the idempotent $e = a \wedge a^{-1}a$.
But if $F \subseteq E(S)$ is an ultrafilter containing $e$ then $F^{\uparrow}$ is an idempotent ultrafilter in $S$ containing $a$.
Thus  $\epsilon (V_{a} \cap \mathsf{G}(S)_{o}) = V_{e}^{E}$.
\end{proof}

The main theorem of this section now follows.

\begin{theorem}[Comparison theorem] Let $S$ be a pre-Boolean inverse $\wedge$-semigroup.
Then $\mathsf{G}(S)$ is a Hausdorff Boolean groupoid, and $\mathsf{D}(S)$ is isomorphic to $\mathsf{B}(\mathsf{G}(S))$. 
\end{theorem}
\begin{proof} The fact that  $\mathsf{G}(S)$ is a Hausdorff Boolean groupoid follows from Proposition~2.33 
and the fact that the semilattice of idempotents of $S$ is a pre-Boolean semilattice
whose associated topological space is Boolean and homeomorphic to the space of identities of $\mathsf{G}(S)$. 
After Lemma~2.14, we defined what we can now see is a map $\beta' \colon \mathsf{D}(S) \rightarrow \mathsf{B}(\mathsf{G}(S))$. 
We proved there that this was an injective homomorphism. 
By Lemma~2.15, this homomorphism is surjective because every element of $\mathsf{B}(\mathsf{G}(S))$,
being compact-open, is a finite union of elements of the form $V_{s}$ where $s \in S$
since these form a compact-open basis.
\end{proof}

%%%%%%%%%%%%%%%%%%%%%%%%%%%%%%%%%%%%%%%%%%%%%%%%%%%%%%%%%%%%%%%%%%%%%%%%%%%%%%%%%%%%%%%%%%%%%%%%%%%%%%%%%%%%%%%%%%%%%%%%%%%%%%%%%%%%%%%%%%%%%%%%%%%%%%%%%%%%%%%%%%%%%%%%
\subsection{The Duality Theorem }

We have seen how to complete pre-Boolean inverse $\wedge$-semigroups to Boolean inverse $\wedge$-semigroups.
We have also seen how to construct Boolean inverse $\wedge$-semigroups from Hausdorff Boolean groupoids.
We now complete the circle of ideas by proving the Duality Theorem.
The following was proved as Lemma~2.17 of \cite{Law11}.

\begin{lemma} Let $\alpha \colon G \rightarrow H$ be a covering functor between groupoids.
Then $\alpha^{-1} \colon \mathsf{Bi}(H) \rightarrow \mathsf{Bi}(G)$ is a morphism of inverse $\wedge$-semigroups and 
defines a contravariant functor from the category of
groupoids and their covering functors to the category of Boolean inverse $\wedge$-semigroups and their morphisms.
\end{lemma}

Let $G$ be a Hausdorff Boolean groupoid.
For each $g \in G$ define
$$\mathcal{F}_{g} = \{A \in \mathsf{B}(G) \colon g \in A \}.$$
The following was proved as Lemma~2.19 of \cite{Law11};
although the definition of Boolean groupoid used in that paper is more restricted than the meaning used here,
the proofs are equally valid for our more general case.

\begin{lemma} With the above definition we have the following.
\begin{enumerate}

\item $\mathcal{F}_{g}$ is an ultrafilter and every ultrafilter in $\mathsf{B}(G)$ is of this form.

\item If $gh$ exists in the groupoid $G$ then $\mathcal{F}_{g} \cdot \mathcal{F}_{h} = \mathcal{F}_{gh}$.

\item $\mathcal{F}_{g}^{-1} = \mathcal{F}_{g^{-1}}$.

\item $\mathcal{F}_{g} = \mathcal{F}_{h}$ if and only if $g = h$.
 
\end{enumerate}
\end{lemma}
\begin{proof}
(1) Let $g \in G$.
It is immediate that $\mathcal{F}_{g}$ is a filter so it only remains to check that it is an ultrafilter.
We use Lemma~2.6(2).
Let $A \in \mathsf{B}(G)$ be a compact-open bisection with the property that $A \cap B \neq \emptyset$ for each $B \in \mathsf{B}(G)$.
We shall prove that $g$ belongs to the closure of $A$; but in a Hausdorff space compact sets are closed and so this will imply
that $g \in A$.
Let $O$ be any open set containing $g$.
By definition there is a compact-open bisection $C$ such that $g \in C \subseteq O$.
But $C$ is a compact-open bisection containing $g$ and so $C \in \mathcal{F}_{g}$.
By assumption $C \cap A \neq \emptyset$.
We have proved that every open set containing $g$ contains elements of $A$.
It follows that $g$ belongs to the closure of $A$, as required.

Now let $F$ be any ultrafilter in $\mathsf{B}(G)$.
We shall prove that $F \subseteq \mathcal{F}_{g}$ for some $g \in G$ from which the claim will follow.
Let $A \in F$ be any compact-open bisection belonging to $F$.
Consider the set $F' = \{A \cap B \colon b \in F \}$,
a subset of $F$ because $F$ is a filter.
In addition, this is a family of closed subsets of $G$, since compact subsets of hausdorff spaces are closed, 
with the property that any finite intersection is non-empty, because $F$ is a filter. 
But each element of $F'$ is a subset of the compact set $A$.
It follows that the set $F'$ has a non-empty intersection.
Let $g$ belong to this intersection.
Then $g$ belongs to every element of $F$ and so $F \subseteq \mathcal{F}_{g}$, as required.

The proofs of (2) and (3) are straightforward, and (4) follows from the fact that the groupoid $G$ is Hausdorff.
\end{proof}

\begin{proposition} Define $\mathsf{B}$ to take the Hausdorff Boolean groupoid $G$ to the  Boolean inverse $\wedge$-semigroup $\mathsf{B}(G)$ 
and the proper continuous covering functor $\alpha \colon G \rightarrow H$ 
to the function $\alpha^{-1} \colon \mathsf{B}(H) \rightarrow \mathsf{B}(G)$.
Then $\mathsf{B}$ is a contravariant functor from the category of Hausdorff Boolean groupoids to the category of Boolean inverse $\wedge$-semigroups.
\end{proposition}
\begin{proof} Let $\alpha \colon G \rightarrow H$ be a proper continuous covering functor between two Boolean groupoids.
By Lemma~2.35 and the fact that $\alpha$ is proper, we have that $\alpha^{-1}$ is a $\wedge$-homomorphism from $\mathsf{B}(H)$ to $\mathsf{B}(G)$.
It remains to prove that $\alpha^{-1}$ pulls ultrafilters in $\mathsf{B}(G)$ back to ultrafilters in $\mathsf{B}(H)$. 
Let $F$ be an ultrafilter in $\mathsf{B}(G)$.
By Lemma~2.36(1), there exists $g \in G$ such that $F = \mathcal{F}_{g}$.
Put $h = \alpha (g)$.
Then $G = \mathcal{F}_{h}$ is an ultrafilter in $H$.
The inverse image of $F$ under $\alpha^{-1}$ consists of all those compact-open local bisections $B$ such that
$\alpha^{-1} (B) \in F$.
But $\alpha^{-1} (B) \in F$ if and only if $g \in \alpha^{-1} (B)$ if and only if $\alpha (g) \in B$
if and only if $B \in G$, as required. 
\end{proof}

The following was proved as Proposition~2.15 of \cite{Law11}.

\begin{lemma} Let $\theta \colon S \rightarrow T$ be a morphism of Boolean inverse $\wedge$-semigroups.
Then $\theta^{-1} \colon \mathsf{G}(T) \rightarrow \mathsf{G}(S)$ is a covering functor.
\end{lemma}

\begin{proposition} Define $\mathsf{G}$ to take the Boolean inverse $\wedge$-semigroup $S$ to the Hausdorff Boolean groupoid $\mathsf{G}(S)$
and the proper morphism $\theta \colon S \rightarrow T$ between Boolean inverse $\wedge$-semigroups to the function
$\theta^{-1} \colon \mathsf{G}(T) \rightarrow \mathsf{G}(S)$.
Then $\mathsf{G}$ defines a contravariant functor from the category of Boolean inverse $\wedge$-semigroups
to the category of Hausdorff Boolean groupoids.
\end{proposition}
\begin{proof}
Let $\theta \colon S \rightarrow T$ be a proper morphism between Boolean inverse $\wedge$-semigroups.
By assumption, $\theta^{-1} \colon \mathsf{G}(T) \rightarrow \mathsf{G}(S)$ is a well-defined function.
Put $\phi = \theta^{-1}$.
By Lemma~2.38, it follows that $\phi$ is a covering functor.
It remains to prove that $\phi$ is continuous and proper.
The basic open sets of $\mathsf{G}(S)$ are of the form $V_{s}$ where $s \in S$.
Put $t = \theta (s)$.
Then $F$ belongs to the inverse image of $V_{s}$ under $\phi$ if and only if
$\theta^{-1}(F) \in V_{s}$ if and only if $s \in \theta^{-1}(F)$ if and only if $\theta (s) \in F$
if and only if $F \in V_{t}$, as required. 

It remains to show that $\theta^{-1}$ is proper.
Let $X$ be a compact subset of $\mathsf{G}(S)$.
Clearly $X \subseteq \bigcup_{s \in S}V_{s}$.
But $X$ is compact and so we can find a finite number of elements of $S$ such that
$X \subseteq \bigcup_{i=1}^{n} V_{s_{i}}$.
It follows that
$\phi^{-1} (X) \subseteq \bigcup_{i=1}^{n} V_{\theta (s_{i})}$.
Now $X$ is a compact subset of a Hausdorff space and so is closed.
It follows that $\phi^{-1}(X)$ is closed.
But $\bigcup_{i=1}^{n} V_{\theta (s_{i})}$ is a finite union of compact subsets and so is compact.
But a closed subset of a compact spaces is compact.
It follows that $\phi^{-1} (X)$ is compact, as claimed.
\end{proof}

%%%%%%%%%%%%%%%%%%%%%%%%%%%%%%%%%%%%%%%%%%%%%%%%%%%%%%%%%%%%%%%%%%%%%%%%%%%%%%%%%%%%%%%%%%%%%%%%%%%%%%%%%%%%%%%%%%%%%%%%%%%%%%%%%%%%%%%%
\begin{center}
{\bf Proof of the Duality Theorem}
\end{center}

\begin{proof}
Let $S$ be a Boolean inverse $\wedge$-semigroup.
Then $\mathsf{G}(S)$ is a Boolean groupoid and $\mathsf{B}(\mathsf{G}(S))$ is a Boolean inverse semigroup.
By Theorem~2.34, the Comparison Theorem, $\mathsf{D}(S)$ is isomorphic to $\mathsf{B}(\mathsf{G}(S))$.
We therefore need to show that $S$ is isomorphic to $\mathsf{D}(S)$.
It is enough to show that if $\theta \colon S \rightarrow T$ is a cover-to-join map to a 
distributive inverse semigroup then $\theta$ is automatically a join-preserving homomorphism.
The isomorphism will then follow from the universal property of $\mathsf{D}(S)$.
Let $s = \bigvee_{i=1}^{n} s_{i}$ in $S$.
Then the $s_{i}$ are pairwise compatible.
We prove that $\{s_{1}, \ldots, s_{n}\}$ is a cover of $s$.
Let $0 \neq a \leq s$.
Then $a = sa^{-1}a$.
It follows that $a = \bigvee_{i=1}^{n} s_{i}a^{-1}a$.
Since $a$ is non-zero one of the $s_{i}a^{-1}a$ must be non-zero.
Thus the idempotent $s_{i}^{-1}s_{i}a^{-1}a$ is non-zero.
But $a$ and $s_{i}$ are both in $s^{\downarrow}$ and so they are compatible.
It follows also that $\mathbf{s_{i} \wedge a} = s_{i}^{-1}s_{i}a^{-1}a$.
Thus $s_{i} \wedge a \neq 0$, as required.
It follows that $\theta (s) = \bigvee_{i=1}^{n} \theta (s_{i})$, as required.

Let $G$ be a Hausdorff Boolean groupoid.
Then $\mathsf{B}(G)$ is a Boolean inverse semigroup and $\mathsf{G}(\mathsf{B}(G))$ is a Hausdorff Boolean groupoid.
The function $g \mapsto \mathcal{F}_{g}$ is a bijective functor by Lemma~2.36.
We need to show that this map is continuous and open.
Let $U$ be a compact-open bisection of $G$.
The image of $U$ under this map is $\{\mathcal{F}_{g} \colon g \in U \}$.
Now $U$ is an element of the inverse semigroup $\mathsf{B}(G))$.
Within this inverse semigroup, we may consider all the ultrafilters that contain $U$ as an element.
That is, the set $V_{U}$.
By Lemma~2.36, ultrafilters in $\mathsf{B}(G))$ all have the form $\mathcal{F}_{h}$ for $h \in G$.
Now the ultrafilter $\mathcal{F}_{h}$ contains $U$ if and only if $h \in U$.
Hence
$$V_{U} = \{\mathcal{F}_{g} \colon g \in U \}$$
is the image of $U$ under our map and so is an open set.

Finally, we prove that our map is continuous.
A basic compact-open subset of $\mathsf{G}(\mathsf{B}(G))$ is of the form $V_{U}$
where $U$ is a compact-open bisection of $G$.
The inverse image of $V_{U}$ under our map is $U$.
\end{proof}

%%%%%%%%%%%%%%%%%%%%%%%%%%%%%%%%%%%%%%%%%%%%%%%%%%%%%%%%%%%%%%%%%%%%%%%%%%%%%%%%%%%%%%%%%%%%%%%%%%%%%%%%%%%%%%%%%%%%%%%%%%%%%%%
\section{Basic properties}

In this section, we shall prove some results that will be used to help us analyse the examples in Section~4.

%%%%%%%%%%%%%%%%%%%%%%%%%%%%%%%%%%%%%%%%%%%%%%%%%%%%%%%%%%%%%%%%%%%%%%%%%%%%%%%%%%%%%%%%%%%%%%%%%%%%%%%%%%%%%%%%%%%%%%%%%%%%%%%
\subsection{Unambiguous $E^{\ast}$-unitary inverse semigroups}

The following lemma will simplify checking when a map is a cover-to-join map.
Note that $\mathcal{D}$ is one of Green's relations.

\begin{lemma} Let $S$ be an inverse $\wedge$-semigroup and let $\theta \colon S \rightarrow T$ be a homomorphism to a distributive inverse semigroup.
\begin{enumerate}

\item The map $\theta$ is a cover-to-join map if and only if the map $\theta \mid E(S) \colon E(S) \rightarrow E(T)$ is a
cover-to-join map.

\item  Let $\{f_{i} \colon i \in I\}$ be an idempotent transversal of the non-zero $\mathcal{D}$-classes of $S$.
Then $\theta$ is a cover-to-join map if and only if it is a cover-to-join map for the distinguished family of idempotents.

\end{enumerate}
\end{lemma} 
\begin{proof} 

(1) Only one direction needs proving: 
we prove that if $\theta \mid E(S) \colon E(S) \rightarrow E(T)$ is a cover-to-join map then $\theta$ is a cover-to-join map.
Let $\{a_{1}, \ldots, a_{m}\}$ be a cover of $a$.
We prove that $\{\mathbf{d}(a_{1}), \ldots, \mathbf{d}(a_{m})\}$ is a cover of $\mathbf{d}(a)$.
Let $0 \neq e \leq \mathbf{d}(a)$.
Put $a' = ae$. 
Then $\mathbf{d}(a') = e$ and so $0 \neq a' \leq a$.
Thus by assumption $a_{i} \wedge a' \neq 0$ for some $i$.
But $a_{i}$ and $a'$ are bounded above by $a$ and so are compatible.
By compatibility, we get that
$\mathbf{d}(a_{i} \wedge a') = \mathbf{d}(a_{i}) \mathbf{d}(a') = e \neq 0$, as required.
By the assumption on $\theta$ it follows that
$$\theta (\mathbf{d}(a)) = \bigvee_{i} \theta (\mathbf{d}(a_{i})).$$ 
Multiply this equality on the left by $\theta (a)$.
The lefthand-side becomes $\theta (a)$ and the righthand side becomes $\bigvee_{i} \theta (a_{i})$, as required. 

(2) Let $e$ be an arbitrary idempotent such that $e \mathcal{D} f$ where $f$ belongs to our transversal.
Let $\{e_{1}, \ldots, e_{m} \}$ be a cover of $e$, where all elements are idempotents.
Let $a$ be any element of $S$ such that $a^{-1}a = e$ and $aa^{-1} = f$.
Put $f_{i} = ae_{i}a^{-1}$ all bounded above by $f$.
We claim that $\{f_{1}, \ldots, f_{m} \}$ is a cover of $f$.
Let $0 \neq j \leq f$.
Then $0 \neq a^{-1}ja \leq e$.
By assumption there exists $e_{i}$ such that $e_{i} \wedge a^{-1}ja \neq 0$. 
Observe that $e(e_{i} \wedge a^{-1}ja) = e_{i} \wedge a^{-1}ja$.
It follows that $a(e_{i} \wedge a^{-1}ja)a^{-1} \neq 0$.
Thus  $a(e_{i} \wedge a^{-1}ja)a^{-1} = j \wedge f_{i} \neq 0$, as required,
where we use the fact that multiplication always distributes over finite meets by Proposition~1.4.19 of \cite{Law1}.
Thus $\{f_{1}, \ldots, f_{m} \}$ is a cover of $f$.
It follows that $\theta (f) = \bigvee_{i} \theta (f_{i})$.
Now multiply this equation on the left by $\theta (a)^{-1}$ and on the right by $\theta (a)$.
This then gives us $\theta (e) = \bigvee_{i} \theta (e_{i})$, as required.
\end{proof}

%%%%%%%%%%%%%%%%%%%%%%%%%%%%%%%%%%%%%%%%%%%%%%%%%%%%%%%%%%%%%%%%%%%%%%%%%%%%%%%%%%%%%%%%%%%%%%%%%%%%%%%%%%%%%%%%%%%%%%%%%%%%%%%%%%%%%%%%%%%%%%%%%%%%%%%%%%%%
We now locate an important class of inverse $\wedge$-semigroups.
An inverse semigroup with zero is is said to be {\em $E^{\ast}$-unitary} if $0 \neq e \leq s$ where $e$ is an idempotent implies that $s$ is an idempotent.
The following is Remark~2.3 of \cite{Lenz} which is worth repeating since it was a surprise to many people not least the author.

\begin{lemma} 
If $S$ is an $E^{\ast}$-unitary inverse semigroup then $S$ is an inverse $\wedge$-semigroup.
\end{lemma}

There are many naturally occurring examples of $E^{\ast}$-unitary inverse semigroups and it is a condition that is easy to verify.
In particular, both the graph inverse semigroups and the tiling semigroups are $E^{\ast}$-unitary.
All the inverse semigroups in Section~4  will be $E^{\ast}$-unitary.

A poset with zero $X$ is said to be {\em unambiguous}\footnote{Strictly speaking `unambiguous except at zero' but that is too much of a mouthful.}
if for all $x,y \in X$ if there exists $0 \neq z \leq x,y$ the either $x \leq y$ or $y \leq x$.
An inverse semigroup with zero $S$ will be said to be {\em unambiguous} if its semilattice of idempotents is unambiguous.
That is, for any two non-zero idempotents $e$ and $f$ if $ef \neq 0$ then $e$ and $f$ are comparable.

Let $E$ be a meet semilattice with zero.
Given $e,f \in E$ we say that $e$ {\em is directly above} $f$ and that $f$ {\em is directly below} $e$ if $e > f$ and there is no $g \in E$ such that $e > g > f$.
For each $e \in E$ define $\hat{e}$ to be the set of elements of $E$ that are directly below $e$.
The semilattice $E$ is said to be {\em pseudofinite} if for each $e \in E$ the set $\hat{e}$ is finite
and whenever $e > f$ there exists $g \in \hat{e}$ such that $e > g > f$.
An inverse semigroup is pseudofinite if its semilattice of idempotents is pseudofinite.
The meet semilattice $E$ is said to satisfy the {\em Dedekind height condition} if for all
non-zero elements $e$ the set  $\left| e^{\uparrow} \right| < \infty$.
An inverse semigroup is said to satisfy this condition if its semilattice of idempotents does.

\begin{proposition} Let $E$ be an unambiguous semilattice with zero which is 
pseudofinite and satisfies the Dedekind finiteness condition.
Then $E$ is pre-Boolean.
\end{proposition}
\begin{proof} We claim first that if $0 \neq f < e$ and if $V_{e:f} \neq \emptyset$
then we may find idempotents $e_{1}, \ldots, e_{m} \leq e$ such that
$$V_{e:f} = \bigcup_{i=1}^{m} V_{e_{i}}.$$

For each non-zero element $e$, the set $\hat{e}$ is a finite orthogonal set by unambiguity.

Let $0 \neq f < e$.
If $V_{e:f} \neq \emptyset$, then $V_{e:f}$ may be written as a finite union of sets of the form $V_{e: j}$
where $j < i$ is immediately below $e$.
By assumption, the set $f^{\uparrow}$ is finite.
There is therefore an element $f < g \leq e$ such that $g$ is immediately above $f$.
Observe that
$$V_{e: f} = V_{g:f} \cup V_{e:g}.$$
The argument is now repeated with $V_{e:g}$ and by induction we have proved the claim.

 Let $0 \neq f < e$ be such that $\hat{e} = \{f,e_{1}, \ldots, e_{m} \}$.
Then if $V_{e:f} \neq  \emptyset$ then
$$V_{e:f} = \bigcup_{i=1}^{m} V_{e_{i}}.$$
Let $F \in V_{e:f}$.
Then $e \in F$ and $f \notin F$.
By Lemma~2.6(2), there exists $g \in F$ such that $g \wedge f = 0$.
Since $g \wedge e \in F$ we may in fact assume that $g \leq e$.
Now $g \leq e$ implies that either $g \leq f$ or $g \leq e_{i}$ for some $i$.
We cannot have the former and so we must have the latter.
It follows that $e_{i} \in F$ and so $F \in V_{e_{i}}$.
To prove the reverse inclusion, let $F \in V_{e_{i}}$.
Then $F \in V_{e}$.
But we cannot have $f \in F$ because $f \wedge e_{i} = 0$.
It follows that $F \in V_{e:f}$.

We have therefore proved that if $V_{e:f} \neq \emptyset$ then
$$V_{e:f} = \bigcup_{i=1}^{m} V_{e_{i}}.$$
The claim is proved when we observe that the equality holds when the $e_{i}$ are replaced by $e_{i} \wedge e$.

%%%%%%%%%%%%%%%%%%%%%%%%%%%%%%%%%%%%%%%%%%%%%%%%%%%%%%%%%%%%%%%%%%%%%%%%%%%%%%%%%%%%%%%%%%%%%%%%%%%%%
We now prove that every tight filter is an ultrafilter which implies that $E$ is pre-Boolean
by part (2) of the Booleanization Theorem.
Let $P$ be a tight filter and suppose that it is not an ultrafilter.
Then there is an ultrafilter $Q$ such that $P \subset Q$.
Let $f \in Q \setminus P$ and let $e \in P$.
Now $e \wedge f \in Q \setminus P$.
It follows that we may choose $f \in Q \setminus P$ such that $f \leq e$.
Now $P$ is a tight filter that contains $e$ and omits $f$.
It follows by Lemma~2.21(1), that there exists $0 \neq g \leq e$ such that $g \wedge f = 0$.
By Lemma~2.6(1), there is an ultrafilter containing $g$ that must therefore omit $f$.
But then we have shown that $V_{e:f} \neq \emptyset$.
By what we proved above we have that
$$V_{e:f} = \bigcup_{i=1}^{m} V_{e_{i}}$$
for some idempotents $e_{1}, \ldots, e_{m}$.
It follows that
$$V_{e} =  \left( \bigcup_{i=1}^{m} V_{e_{i}} \right) \cup V_{f}.$$
Thus by Lemma~2.29, we have that
$$e \rightarrow (e_{1}, \ldots, e_{m},f).$$ 
Now $e \in P$ and $P$ is a tight filter.
It follows that either $f \in P$ or $e_{i} \in P$ for some $i$.
By construction, $f \notin P$.
It follows that $e_{i} \in P$ for some $i$.
It follows that $f,e_{i} \in Q$.
But $Q$ is an ultrafilter containing $e_{i}$.
It follows that $F \in V_{e:f}$ and so $f \notin Q$.
We have therefore arrived at a contradiction.
It follows that every tight filter is an ultrafilter.
\end{proof}

Concrete examples of semilattices satisfying the conditions of the above proposition may easily be constructed.
Let $G$ be a directed graph and let $G^{\ast}$ denote the free category generated by $G$.
Denote by $E$ the set of all elements of $G^{\ast}$ together with a zero element.
Define $e \leq f$ if and only if $e = fg$ for some element $g \in G^{\ast}$;
in other words, the prefix ordering.
With respect to this order, 
$E$ becomes a semilattice which is unambiguous and satisfies the Dedekind finiteness condition.
The following are easy to check.

\begin{lemma} With the above definitions, we have the following.
\begin{enumerate}

\item The semilattice $E$ has no $0$-minimal idempotents if and only if the in-degree of each vertex is at least 1.

\item The semilattice $E$ is $0$-disjunctive if and only if the in-degree of each vertex is either zero or at least 2.

\item  The semilattice $E$ is pseudofinite if and only if the in-degree of each vertex is finite.

\end{enumerate}
\end{lemma}

%%%%%%%%%%%%%%%%%%%%%%%%%%%%%%%%%%%%%%%%%%%%%%%%%%%%%%%%%%%%%%%%%%%%%%%%%%%%%%%%%%%%%%%%%%%%%%%%%%%%%%%%%%%%%%%%%%%%%%%%%%%%%%%%%%%%%%%%
\subsection{The group of units of $\mathsf{D}(S)$}

Let $S$ be a pre-Boolean inverse $\wedge$-semigroup.
In this section, we shall be interested in calculating the group of units of $\mathsf{D}(S)$ in the case where it is unital.
By Theorem~2.28, this occurs precisely when $E(S)$ is compactable.   
We shall focus on  the case where $S$ is unambiguous and $E^{\ast}$-unitary.
The proofs of the following can be found in \cite{JL2} or easily proved directly.

\begin{lemma} \mbox{}
\begin{enumerate}

\item Let $S$ be an unambiguous inverse semigroup.
Then $(S,\leq)$ is an unambiguous poset if and only if $S$ is $E^{\ast}$-unitary. 

\item Let $S$ be an unambiguous $E^{\ast}$-unitary inverse semigroup.
Then $S$ is separative if and only if the semilattice of idempotents $E(S)$ is $0$-disjunctive.

\end{enumerate}
\end{lemma}

\begin{lemma} Let $S$ be an unambiguous $E^{\ast}$-unitary inverse semigroup.
Given a finitely generated order ideal $\{s_{1}, \ldots, s_{m} \}^{\downarrow}$,
there exists an orthogonal set $\{t_{1}, \ldots, t_{n}\}$ such that
$\{s_{1}, \ldots, s_{m} \}^{\downarrow} = \{t_{1}, \ldots, t_{n} \}^{\downarrow}$.
\end{lemma}
\begin{proof}
Starting with $i = 1$ compare $s_{i}$ with $s_{j}$ where $j > i$.
If they are orthogonal continue; if $s_{i} < s_{j}$ then discard $s_{i}$ and increase $i$ by 1;
if $s_{j} < s_{i}$ then discard $s_{j}$ and continue comparing.
\end{proof}

The following is now immediate.

\begin{corollary} Let $S$ be an unambiguous $E^{\ast}$-unitary inverse semigroup.
Then every element of $\mathsf{D}(S)$ is an orthogonal join of elements from $\delta (S)$.
\end{corollary} 

A homomorphism of inverse semigroups is said to be {\em idempotent-pure} if the inverse images of idempotents consist only of idempotents.

\begin{lemma} Let $S$ be $E^{\ast}$-unitary.
Then the congruence $\equiv$ defined on $\mathsf{FC}(S)$ is idempotent-pure.
\end{lemma}
\begin{proof} Suppose that $(a_{1}, \ldots, a_{m}) \leftrightarrow (e_{1}, \ldots, e_{n})$ where the $e_{j}$ are idempotents.
Then the $a_{i}$ are idempotents.
This follows because for each $i$ there exists $j$ such that $a_{i} \wedge e_{j} \neq 0$.
But $a_{i} \wedge e_{j}$ is a non-zero idempotent beneath $a_{i}$ and so $a_{i}$ is also an idempotent.
\end{proof}

A finite compatible subset $A \subseteq S$ is said to be {\em essential} when both
$\mathbf{d}(A) = \{ a^{-1}a \colon a \in A\}$ and $\mathbf{r}(A) = \{aa^{-1} \colon a \in A \}$
are essential sets of idempotents; that is, both are covers of $E(S) \setminus \{ 0\}$.

Let $T$ be an inverse semigroup with zero.
An idempotent $e$ in $T$ is said to be {\em essential}
if for each non-zero idempotent $f \in S$ we have that $f \wedge e \neq 0$.
This says that $\{e\}$ is a cover of $E(T)$.
We denote by $T^{e}$ the set of all elements $s \in T$ such that $\mathbf{d}(s)$ and $\mathbf{r}(s)$ are essential idempotents.
By Lemma~4.2 of \cite{Law6}, $T^{e}$ is an inverse subsemigroup of $T$.

\begin{lemma} Let $S$ be separative. 
An idempotent $\{e_{1}, \ldots, e_{n}\}^{\downarrow}$ in $\mathsf{FC}(S)$ is essential if and only if
the set $\{e_{1}, \ldots, e_{n} \}$ is essential.
\end{lemma} 
\begin{proof} Suppose that $\{e_{1}, \ldots, e_{n}\}^{\downarrow}$ is an essential idempotent.
Let $f \in E(S)$ be any non-zero idempotent.
Then $f^{\downarrow}\{e_{1}, \ldots, e_{n}\}^{\downarrow}$ is a non-zero idempotent.
Thus for some $i$ we have that $fe_{i}$ is non-zero, as required.

Suppose that $\{e_{1}, \ldots, e_{n} \}$ is a cover of $E(S)$.
Then it is easy to check that  $\{e_{1}, \ldots, e_{n}\}^{\downarrow}$ is an essential idempotent.
\end{proof}

On an inverse semigroup $S$ define $s \, \sigma \, t$ if and only if there exists $u \leq s,t$.
Then $\sigma$ is a congruence and $S/\sigma$ is a group.
It is, in fact, the {\em minimum group congruence}.  
See Section~2.4 of \cite{Law1}.

\begin{theorem} Let $S$ be $E^{\ast}$-unitary, unambiguous, compactable and separative.
Then the group of units of $\mathsf{D}(S)$ is isomorphic to the group $\mathsf{FC}(S)^{e}/\sigma$.
\end{theorem}
\begin{proof} If $A \in \mathsf{FC}(S)$ then we denote the $\equiv$-class containing $A$ by $[A]$.
Since $A$ is a finitely generated order ideal, we have that $A = X^{\downarrow}$ where $X$ is a finite compatible subset of $S$.
By definition, the element $[A]$ is in the group of units of $\mathsf{D}(S)$ 
if and only if 
$[A]^{-1}[A]$ and $[A][A]^{-1}$ are both the identity element of $\mathsf{D}(S)$.
But $A$ is a compatible order ideal and so we have that $A^{-1}A = \{\mathbf{d}(a) \colon a \in A\} = \mathbf{d}(A)$ and similarly for $AA^{-1}$.
Thus both $[\mathbf{d}(A)]$ and $[\mathbf{r}(A)]$ are equal to the identity
and so both $\mathbf{d}(X)$ and $\mathbf{r}(X)$ are essential sets of idempotents by Theorem~2.28.
By definition, it follows that $X$ is an essential subset of $S$.
By Lemma~3.6, we may assume that the set $X$ is orthogonal.

Let $[A] = [X^{\downarrow}]$ and $[B] = [Y^{\downarrow}]$ be two invertible elements of $\mathsf{D}(S)$
where both $X$ and $Y$ are finite, orthogonal essential sets.
We shall prove that $A \equiv B$ if and only if $A \, \sigma \, B$.
First we need an auxiliary result.

Define $\{a_{1}, \ldots, a_{m}\} \preceq \{b_{1}, \ldots, b_{n}\}$ 
if and only if
$\{a_{1}, \ldots, a_{m} \}^{\downarrow} \subseteq \{b_{1}, \ldots, b_{n} \}^{\downarrow}$
and 
$\{b_{1}, \ldots, b_{n}\} \rightarrow  \{a_{1}, \ldots, a_{m}\}$.
Now observe that if $\{a_{1}, \ldots, a_{m}\} \leftrightarrow \{b_{1}, \ldots, b_{n}\}$
then
$$\{a_{1}, \ldots, a_{m}\} \rightarrow \{a_{i} \wedge b_{j} \colon 1 \leq i \leq m, 1 \leq j \leq n\}$$
and
$$\{b_{1}, \ldots, b_{n}\} \rightarrow  \{a_{i} \wedge b_{j} \colon 1 \leq i \leq m, 1 \leq j \leq n\}.$$
It follows that if  $\{a_{1}, \ldots, a_{m} \} \leftrightarrow \{b_{1}, \ldots, b_{n}\}$
then there is $\{c_{1}, \ldots, c_{p}\}$ such that
$\{c_{1}, \ldots, c_{p}\} \preceq \{a_{1}, \ldots, a_{m}\}$ 
and
$\{c_{1}, \ldots, c_{p}\} \preceq \{b_{1}, \ldots, b_{n}\}$.
On the other hand, 
if  $\{a_{1}, \ldots, a_{m}\} \preceq \{b_{1}, \ldots, b_{n}\}$ then  in fact 
$\{a_{1}, \ldots, a_{m}\} \leftrightarrow \{b_{1}, \ldots, b_{n}\}$.
Thus 
$$\{a_{1}, \ldots, a_{m}\} \leftrightarrow \{b_{1}, \ldots, b_{n}\}$$
if and only if 
there is $\{c_{1}, \ldots, c_{p}\}$ such that
$$\{c_{1}, \ldots, c_{p}\} \preceq \{a_{1}, \ldots, a_{m}\}
\text{ and }
\{c_{1}, \ldots, c_{p}\} \preceq \{b_{1}, \ldots, b_{n}\}.$$

It is immediate by the calculation above that if $A \equiv B$ then $A \, \sigma \, B$.
To prove the converse suppose that $A \, \sigma \, B$.
Then there is an essential set $C$ such that $C \subseteq A,B$.
We prove that if $C \subseteq A$ and both $C$ and $A$ are essential then $C \equiv A$.
Let $A = \{a_{1}, \ldots, a_{m}\}$ and $C = \{c_{1}, \ldots, c_{p} \}$.
It is enough to prove that
$(a_{1}, \ldots, a_{m}) \rightarrow (c_{1}, \ldots, c_{p})$.
Let $0 \neq x \leq a_{i}$.
Then $0 \neq \mathbf{d}(x) \leq \mathbf{d}(a_{i})$.
Thus $\mathbf{d}(x) \wedge \mathbf{d}(c_{j}) \neq 0$ for some $j$.
Now $c_{j} \leq a_{k}$ for some $k$.
But $\mathbf{d}(x) \leq \mathbf{d}(a_{i})$ and $\mathbf{d}(c_{j}) \leq \mathbf{d}(a_{k})$.
Thus $0 \neq \mathbf{d}(x) \wedge \mathbf{d}(c_{j}) \leq \mathbf{d}(a_{i}) \wedge \mathbf{d}(a_{k})$.
But $a_{i}$ and $a_{k}$ are assumed orthogonal if they are not equal.
Thus $a_{i} = a_{k}$.
We deduce that we must have $x,c_{j} \leq a_{i}$.
But then $x$ and $c_{j}$ are compatible and so $\mathbf{d}(x \wedge c_{j}) = \mathbf{d}(x) \wedge \mathbf{d}(c_{j})$.
It follows that $x \wedge c_{j} \neq 0$, as required.
\end{proof}

%%%%%%%%%%%%%%%%%%%%%%%%%%%%%%%%%%%%%%%%%%%%%%%%%%%%%%%%%%%%%%%%%%%%%%%%%%%%%%%%%%%%%%%%%%%%%%%%%%%%%%%%%%%%%%%%%%%%%%%%%%%%%%%%%%
\subsection{When is $\mathsf{D}(S)$ congruence-free?}

On an inverse semigroup $S$ define $s \, \mu \, t$ if and only if $ses^{-1} = tet^{-1}$ for all idempotents $e \in S$.
Then $\mu$ is a congruence and the natural homomorphism from $S$ to $S/\mu$ is injective when restricted to
the semilattice of idempotents.
Such homomorphisms are said to be {\em idempotent-separating} and the congruence $\mu$ is thus idempotent separating.
It is, in fact, the {\em maximum idempotent-separating congruence} on $S$.
Semigroups for which $\mu$ is equality are called {\em fundamental}.
See Section~5.2 of \cite{Law1}.
An inverse semigroup with zero is said to be {\em $0$-simple} if it has no, non-trivial  ideals.
The following is proved in \cite{Petrich}.

\begin{lemma} An inverse semigroup with zero $S$ is congruence-free if and only if it is 
fundamental, $0$-simple and $E(S)$ is $0$-disjunctive.
\end{lemma}

\begin{lemma} Let $S$ be an unambiguous $E^{\ast}$-unitrary pre-Boolean inverse semigroup.
If $S$ is fundamental and $E(S)$ is $0$-disjunctive 
then $\mathsf{D}(S)$ is fundamental and $E(\mathsf{D}(S))$ is $0$-disjunctive.
\end{lemma}
\begin{proof} By Lemma~3.5(2), the semigroup $S$ is separative.
It follows that $\delta \colon S \rightarrow \mathsf{D}(S)$ is an embedding.
We shall identify $S$ with its image.
Put $T = \mathsf{D}(S)$.
Each element of $T$ can be written as a join of a finite number of elements of $S$.
We may assume that this join is orthogonal by Corollary~3.7.

Let $t$ be a non-idempotent element of $T$.
By assumption, $t$ is a join of a finite number of elements $s_{1}, \ldots, s_{n}$.
Not all of these elements can be idempotent else their join would be idempotent.
Without loss of generality assume that $s_{1}$ is not an idempotent.
By assumption $S$ is fundamental and there exists an idempotent $e \in S$ such that $s_{1}e \neq es_{1}$.
We claim that $et \neq te$.
Suppose that $et = te$.
Then $\bigvee_{i} es_{i} = \bigvee_{j} s_{j}e$.
Multiply this equation on the right by $\dom (s_{1})$.
Then we get $es_{1} = s_{1}e$, a contradiction.
It follows that $et \neq te$ and so $t$ is not in $Z(E(T))$.
Thus the centralizer of the idempotents just consists of idempotents and so $\mathsf{D}(S)$ is fundamental.

It is immediate that $E(\mathsf{D}(S))$ is $0$-disjunctive.
\end{proof}

We shall say that an inverse semigroup $S$ is {\em sufficiently branching} if for every non-zero idempotent $e$
and every $n \geq 2$ we can find non-zero idempotents $e_{1}, \ldots e_{n} \leq e$ which are pairwise orthogonal.
The proof of the following is straightforward.

\begin{lemma} Let $E$ be a semilattice which is $0$-disjunctive and contains no $0$-minimal idempotents.
Then $E$ is sufficiently branching.
\end{lemma}

\begin{lemma} Let $S$ be an unambiguous $E^{\ast}$-unitary pre-Boolean inverse semigroup with a $0$-disjunctive semilattice of idempotents which is sufficiently branching.
If $S$ is $0$-simple then $\mathsf{D}(S)$ is $0$-simple.
\end{lemma}
\begin{proof} By Lemma~3.5, the semigroup $S$ is separative.
We may therefore identify it with its image in $\mathsf{D}(S)$.
Put $T =  \mathsf{D}(S)$.
Let $e$ and $f$ be arbitrary non-zero idempotents in $T$.
Then we can write $e = \bigvee_{i=1}^{n} e_{i}$.
By assumption we can find $n$ orthogonal idempotents $f_{i} \leq f$.
We now use the fact that $S$ is $0$-simple to find $n$ idempotents $g_{i}$ and elements $s_{i}$ such that
$e_{i} \stackrel{s_{i}}{\longrightarrow} g_{i} \leq f_{i}$.
Put $s = \bigvee_{i=1}^{n} s_{i}$.
Then $e \stackrel{s}{\longrightarrow} \bigvee_{i=1}^{n} f_{i} \leq f$.
\end{proof}

We may sum up what we have found in the following.

\begin{theorem} 
Let $S$ be an unambiguous $E^{\ast}$-unitary pre-Boolean inverse semigroup having no $0$-minimal idempotents. 
Then if $S$ is congruence-free so too is $\mathsf{D}(S)$.
\end{theorem}
\begin{proof} By Lemma~3.11, $E(S)$ is $0$-disjunctive and so by Lemma~3.5(2), $S$ is separative.
By Lemma~3.13, $E(S)$ is sufficiently branching.
By Lemmas~3.11 and Lemma~3.14, $\mathsf{D}(S)$ is $0$-simple.
By Lemma~3.12, $\mathsf{D}(S)$ is fundamental and $E(\mathsf{D}(S))$ is $0$-disjunctive.
It follows by Lemma~3.11, that $\mathsf{D}(S)$ is congruence-free.
\end{proof}

%%%%%%%%%%%%%%%%%%%%%%%%%%%%%%%%%%%%%%%%%%%%%%%%%%%%%%%%%%%%%%%%%%%%%%%%%%%%%%%%%%%%%%%%%%%%%%%%%%%%%%%%%%%%%%%%%%%%%%%%%%%%%%%
\section{Examples}

In this section, we shall describe some  examples of pre-Boolean inverse $\wedge$-semigroups and their Boolean completions.
In particular, we show that the Thompson-Higman groups arise from the theory described in this paper,
where the elements of the groups are obtained by `glueing together' elements of a suitable inverse semigroup.

%%%%%%%%%%%%%%%%%%%%%%%%%%%%%%%%%%%%%%%%%%%%%%%%%%%%%%%%%%%%%%%%%%%%%%%%%%%%%%%%%%%%%%%%%%%%%%%%%%%%%%%%%%%%%%%%%%%%%%%%%%%%%%%%%%%%%%%%%%
\subsection{The polycyclic monoids and the Thompson groups $G_{n,1}$}

In this section we shall work with the polycyclic inverse monoids and show that work by the author in \cite{Law6,Law7,Law8} 
can be viewed as a special case of the general theory of this paper.

Let $A = \{a_{1}, \ldots, a_{n} \}$ be an alphabet with $n$ letters.
A string in $A^{\ast}$, the free monoid generated by $A_{n}$, will be called {\em positive}.
The empty string is denoted $\varepsilon$.
If $u = vw$ are strings, then $v$ is called a {\em prefix} of $u$,
and a {\em proper prefix} if $w$ is not the empty string.
A pair of elements of $A_{n}^{\ast}$ is said to be {\em
  prefix-comparable} if one is a prefix of the other.
If $x$ and $y$ are prefix-comparable we define
$$x \wedge y = \left\{ \begin{array}{ll} 
x & \mbox{if $y$ is a prefix of $x$}\\
y &  \mbox{if $x$ is a prefix of $y$}
\end{array}
\right.
$$
The {\em polycyclic monoid $P_{n}$}, where $n \geq 2$,
is defined as a monoid with zero by the following presentation
$$P_{n} = \langle a_{1}, \ldots, a_{n}, a_{1}^{-1}, \ldots, a_{n}^{-1}
\colon \: a_{i}^{-1}a_{i} = 1 \,\mbox{and}\, a_{i}^{-1}a_{j} = 0, i \neq j \rangle.$$ 
Every non-zero element of $P_{n}$ is of the form $yx^{-1}$ where
$x,y \in A^{\ast}$.
Identify the identity with $\varepsilon \varepsilon^{-1}$.
The product of two elements $yx^{-1}$ and $vu^{-1}$ is zero unless
$x$ and $v$ are prefix-comparable.
If they are prefix-comparable then
$$yx^{-1} \cdot vu^{-1} = \left\{
\begin{array}{ll}
yzu^{-1}   & \mbox{if $v = xz$ for some string $z$}\\
y(uz)^{-1} & \mbox{if  $x = vz$ for some string $z$}
\end{array}
\right.
$$
The polyclic monoid $P_{n}$ is an inverse monoid with zero:
the inverse of $xy^{-1}$ is $yx^{-1}$;
the non-zero idempotents in $P_{n}$ are the elements of the form $xx^{-1}$;
the natural partial order is given by  $yx^{-1} \leq vu^{-1}$ iff $(y,x) = (v,u)p$ for some positive string $p$.

Polycyclic inverse monoids are unambiguous, $E^{\ast}$-unitary and congruence-free.
By Lemma~3.5(2) they are therefore separative.
The semilattice of idempotents of $P_{n}$ is the regular $n$-tree and so by Proposition~3.3
polycyclic monoids are also pre-Boolean.
An inverse semigroup is {\em combinatorial} if $s^{-1}s = t^{-1}t$ and $ss^{-1} = tt^{-1}$ implies that $s = t$.
An inverse semigroup with zero $S$ is said to be  {\em $0$-bisimple} if 
for any two non-zero idempotents $e$ and $f$ there exists an element $s$ such that $e = ss^{-1}$ and $f = ss^{-1}$.

\begin{proposition} 
The polycyclic inverse monoids are combinatorial, $0$-bisimple, unambiguous, $E^{\ast}$-unitary, congruence-free, pre-Boolean and separative.
\end{proposition}

It follows by the Completion Theorem that $P_{n}$ has a Boolean completion 
which we denote by $C_{n}$ and call the {\em Cuntz inverse monoid}.\footnote{This term is frequently used, particularly by those working in $C^{\ast}$-algebra theory, 
for the polycyclic monoids themselves.
It seems appropriate to make a distinction as we have done here.}
By Theorem~3.15, the Cuntz monoid is itself congruence-free.
To determine the group of units of $C_{n}$, 
we use Theorem~3.10 which makes a direct connection with the calculations carried out in \cite{Law7}.
We therefore have the following.

\begin{theorem} The Boolean completion of the polycyclic inverse monoid $P_{n}$ is the Cuntz inverse monoid $C_{n}$.
The Cuntz inverse monoid is congruence-free and its group of units is the Thompson group $G_{n,1}$.
\end{theorem}

By Lemma~3.1 and the theory developed in \cite{Law7}, 
a homomorphism $\theta \colon P_{n} \rightarrow T$ where $T$ is a distributive inverse {\em monoid} is a cover-to-join map 
if and only if $\bigvee_{i=1}^{n} \theta (a_{i}a_{i}^{-1})$ is the identity of $T$.
In fact, the covers of the identity of $P_{n}$ correspond bijectively to the maximal prefix codes in the free monoid
on $n$ generators, and all maximal prefix codes may be constructed from the simplest one $\{a_{1}, \ldots, a_{n}\}$ using order-theoretic
and algebraic operations \cite{Law6,Law7}.
It follows that we may describe $C_{n}$ in the following way:
\begin{itemize}

\item It is distributive.

\item It contains a copy of $P_{n}$ and every element of $C_{n}$ is the join of a finite subset of $P_{n}$.

\item $1 = \bigvee_{i=1}^{n} a_{i}a_{i}^{-1}$.

\item It is the freest inverse semigroup satisfying the above conditions.

\end{itemize}

There is a further consequence of this characterization.
The symmetric inverse monoids $I(X)$ are distributive inverse monoids.
A map from $P_{n}$ to $I(X)$ which is a cover-to-join map is precisely what we have called
a {\em strong representation} of $P_{n}$ in the papers \cite{JL1,Law8}.
They lead to isomorphic representations of $C_{n}$ inside $I(X)$.
The theory of such strong representations is, as shown in \cite{JL1},
precisely what the monograph \cite{B} by Bratteli and Jorgensen is about.
In addition, strong representations of $P_{n}$ lead to isomorphic representations of the Thompson group $G_{n,1}$.
An example of such a representation motivated by linear logic is described in Section~9.3 of \cite{Law1}.

%%%%%%%%%%%%%%%%%%%%%%%%%%%%%%%%%%%%%%%%%%%%%%%%%%%%%%%%%%%%%%%%%%%%%%%%%%%%%%%%%%%%%%%%%%%%%%%%%%%%%%%%%%%%%%%%%%%%%%%
\subsection{The polycyclic monoids and the Thompson-Higman groups $G_{n,r}$}

We shall now show how to obtain the remaining Thompson groups from a generalization of the polycyclic inverse monoids.

An inverse semigroup $S$ is said to have {\em maximal idempotents} if for each non-zero idempotent $e$ there is an
idempotent $e^{\circ}$ such that $e \leq e^{\circ}$ where $e^{\circ}$ is a maximal idempotent such that if 
$e \leq i^{\circ}, j^{\circ}$ then $i^{\circ} = j^{\circ}$.
The set of idempotents $\{ e^{\circ} \colon e \in E(S)^{\ast} \}$ is called the set of {\em maximal idempotents}.
The proof of the following is immediate.

\begin{lemma} Let $S$ be an inverse semigroup with a finite number of maximal idempotents $F$.
Then $F$ is a cover of $E(S) \setminus \{ 0\}$. 
\end{lemma}

We shall need the Rees matrix construction in a very simple form.
Let $M$ be any inverse monoid with zero,
and let $r$ be any (finite) cardinal.
Define $B_{r}(M)$ as follows.
The non-zero elements are triples $(i,m,j)$ where
$1 \leq i,j \leq r$ and $m \in M$ is non-zero,
together with a zero $\mathbf{0}$.
The product
$$(i,m,j)(k,n,l) = (i,mn,l)$$
if $j = k$ and $mn$ is non-zero;
otherwise it is $\mathbf{0}$.
Observe that $B_{1}(M)$ is isomorphic to $M$.
The proofs of the following are routine.

\begin{lemma} The semigroup $B_{r}(M)$ is inverse with a semilattice of idempotents which is a $0$-direct union of $r$ copies of the
semilattice of $S$. In addition, we have the following:
\begin{enumerate}

\item If $M$ is $0$-bisimple then so too is $B_{r}(M)$. 

\item If $M$ is unambiguous then so too is $B_{r}(M)$. 

\item If $M$ is $E^{\ast}$-unitary then so too is $B_{r}(M)$. 

\item The semigroup $B_{r}(M)$ has $r$ maximal idempotents.

\end{enumerate}
\end{lemma}

Let $n \geq 2$ and $r \geq 1$ be any finite positive integers.
Define
$$P_{n,r} = B_{r}(P_{n})$$
and call it the {\em extended polycyclic inverse semigroup with parameters $(n,r)$}.

By Lemmas~4.3 and 4.4 together with Proposition~3.3 and Theorem~3.15, we have the following.

\begin{theorem} 
The extended polycyclic inverse semigroup $P_{n,r}$ is a compactable pre-Boolean inverse semigroup.
Its Boolean completion, denoted by $C_{n,r}$, is a congruence-free Boolean inverse $\wedge$-monoid.
\end{theorem}

Our main theorem is the following.

\begin{theorem} 
The group of units of $C_{n,r}$ is the Higman-Thompson group $G_{n,r}$.
\end{theorem}
\begin{proof} 
We shall need some definitions.
If $X$ is a set on which the monoid $S$ acts on the right then $X$ is called a {\em right $S$-act}.
Our actions will always be on the right.
{\em Subacts} are defined in the obvious way and {\em cyclic subacts}
have the form $xS$ where $x \in X$.
Subacts with the property that their intersection with any non-empty subact is also non-empty are said to be {\em essential}.
The act $X$ is said to be {\em finitely generated} if $X = \bigcup_{i=1}^{n} x_{i}S$
for some elements $x_{i} \in X$.
If $X$ is any set and $S$ is any monoid then $S$ acts on
the set $X \times S$ by right multiplication on the second component, 
and this determines what is called a {\em free $S$-act}.
{\em (Right) homomorphisms} between right $S$-acts are defined in the obvious way.

We begin by analysing the construction of the Higman-Thompson groups $G_{n,r}$ by Scott \cite{Scott92}.
Let $X = \{x_{1}, \ldots, x_{r} \}$ and $A = \{a_{1}, \ldots, a_{n}\}$ where we assume that $X \cap A = \emptyset$.
We denote by $A^{\ast}$ the free monoid on $A$.
Scott considers the set $XA^{\ast}$ on which $A^{\ast}$ acts on the right in the obvious way.
The elements of $X$ are simply playing the role of indices, so we could equally well
write $XA^{\ast}$ as $X \times A^{\ast}$ with the action given by $(x,w)a = (x,wa)$.
Thus the starting point for Scott's work is free actions of free monoids.
Because $X$ has $r$ elements, it is a free $r$-generated action.

Scott now goes on to consider subsets, subspaces in Scott's terminology, of the free action which are closed under the action of the free monoid.
These are subacts.
She singles out those subacts which are `cofinite' and `inescapable'.
Birget \cite{Birget} noted that these two conditions together translate into `finitely generated' and `essential'.
Thus we are interested in those finitely generated subacts which are essential.
We consider the set of all isomorphisms between the finitely generated essential subacts.
These form an inverse monoid.
In fact, an inverse monoid with a special property:
each element sits beneath a unique maximal element.
Such inverse monoids are said to be $F$-inverse.
It is this property that enables us to define a group.
Let $G_{n,r}$ be the set of maximal isomorphisms between the finitely generated
essential subacts.
If $f$ and $g$ are two such maps, their composition $fg$ will not in general
be a maximal map but will sit beneath a unique such map.
We define this to be the product of $f$ and $g$, and in this way $G_{n,r}$ is a group.
Alternatively, this group is also the maximum group image of the inverse semigroup of isomorphisms.
Thus Thompson's group $G_{n,r}$ is constructed from
the free $r$-generated action of the free monoid on $n$ generators
by considering the isomorphisms between essential finitely generated subacts.

We now link this definition of the Thompson group $G_{n,r}$ to the inverse semigroup $P_{n,r}$.

%The structure of 0-bisimple inverse {\em monoids} is determined
%by a left cancellative monoid satisfying {\em Clifford's condition}: the intersection of any two principal right ideals
%is either empty or again a principal right ideal.
%Given such a monoid $S$ the set $B(S)$ of all
%right ideal isomorphisms between the principal right ideals together
%with the empty function is a
%$0$-bisimple inverse monoid and every $0$-bisimple inverse monoid is
%obtained in this way. 
%Observe that Clifford's condition is essentially an arithmetic condition on a left cancellative monoid $S$.
%If $z \in xS$ then $z = xu$ and so $z$ is a {\em right multiple} of $x$.
%We denote this by $x|z$.
%The elements of $xS \cap yS$ are therefore the {\em common right multiples} of $x$ and $y$.
%An element $z$ is called a {\em least common right multiple} of $s$ and $t$ if
%$x|z$ and $y|z$, and if $x|z'$ and $y|z'$ then $z|z'$;
%the element $z$ is not unique, but if $z$ and $\bar{z}$ are least common right multiples of $x$ and $y$
%then $\bar{z} = zu$ where $u$ is invertible since $S$ is left cancellative. 
%Thus Clifford's condition says that if $s$ and $t$ have a common right multiple they have
%a least common right multiple unique up to multiplication by a unit.
%The polycyclic inverse monoid $P_{n}$ is $0$-bisimple
%and constructed from the free monoid on $n$ generators in exactly the way described above.

We describe first how to construct all $0$-bisimple inverse semigroups.
Let $X$ be a set on which the monoid $S$ acts on the right.
We say that $(X,S)$ is an {\em RP-system} if $S$ is left cancellative,
the intersection of any two cyclic subacts of $X$ is either empty or again a cyclic subact,
and if $xs = x's$ then $x = x'$.
Define $B(X,S)$ to be the set of all right $S$-isomorphisms between the cyclic subacts together with the empty function.
Then $B(X,S)$ is a $0$-bisimple inverse semigroup and every $0$-bisimple inverse semigroup arises in this way.
This is a classical theorem of semigroup theory which is described in a wider context with references in \cite{Law2}.
A left cancellative monoid  satisfies {\em Clifford's condition} if the intersection of any two principal right ideals
is either empty or again a principal right ideal.
If $S$ is such a monoid then $S$ acts on itself on the right and $B(S,S)$ is a $0$-bisimple inverse {\em monoid},
and every $0$-bisimple inverse {\em monoid} is isomorphic to one constructed in this way.
Let $S$ be a left cancellative monoid satisfying Clifford's condition, whose associated $0$-bisimple inverse monoid is $B(S) = B(S,S)$.
Let $(X \times S,S)$ be a free $S$-system. 
Then it is an easy exercise to check that $B(X \times S,S)$ is isomorphic to $B_{|X|}(B(S))$.
It follows from the above discussion that the extended polycyclic semigroup $P_{n,r}$ 
is constructed from the free system $(\{1,\ldots, r\} \times A_{n}^{\ast},A_{n}^{\ast})$.

Consider now the $RP$-system $(X,S)$ where $X$ is finitely generated.
Then the inverse monoid $D^{e}(B(X,S))$ is isomorphic to the inverse monoid of isomorphisms between the finitely generated essential subacts of $X$.
It follows that when $X = \{1,\ldots, r\} \times A_{n}^{\ast}$ and $S = A_{n}^{\ast}$
then  $D^{e}(B(X,S))/\sigma$ is the Thompson group $G_{n,r}$.
Thus by Theorem~3.10, the group of units of $C_{n,r}$ is also the Thompson group $G_{n,r}$. 
\end{proof}

%%%%%%%%%%%%%%%%%%%%%%%%%%%%%%%%%%%%%%%%%%%%%%%%%%%%%%%%%%%%%%%%%%%%%%%%%%%%%%%%%%%%%%%%%%%%%%%%%%%%%%%%%%%%%%%%%%%%%%%%%%
\subsection{Self-similar groups}

The theory of self-similar groups has been developed by Grigorchuk and his school \cite{N}.
However, in \cite{Law9}, the author showed that they had in fact first been defined by Perrot \cite{P}
as part of a generalization of the pioneering work by Rees \cite{Rees}.
Define a \emph{Perrot semigroup} to be an inverse semigroup that is unambiguous and has the Dedekind height property.
Then self-similar group actions correspond to $0$-bisimple Perrot monoids.
They can be regarded informally as polycyclic monoids which have acquired a group of units.
Those $0$-bisimple Perrot monoids which are also $E^{\ast}$-unitary are pre-Boolean.
This corresponds to the group $G$ acting on the free monoid $A_{n}^{\ast}$ in such a way
that the functions $\phi_{x} \colon G_{x} \rightarrow G$ are all injective, where $\phi_{x}(g) = g|_{x}$. 
The self-similar action is faithful if and only if the inverse monoid is fundamental.
It follows that in this case the inverse monoid is congruence-free.
In this way, certain kinds of self-similar group actions can be used to construct congruence-free Boolean inverse monoids
generalizing the Cuntz inverse monoids.
In fact, arbitrary self-similar group actions are pre-Boolean \cite{LL}.

%%%%%%%%%%%%%%%%%%%%%%%%%%%%%%%%%%%%%%%%%%%%%%%%%%%%%%%%%%%%%%%%%%%%%%%%%%%%%%%%%%%%%%%%%%%%%%%%%%%%%%%%%%%%%%%%%%%%%%%%%%%%%%%
\subsection{Graph inverse semigroups}

Graph inverse semigroups are constructed as a special case of a general procedure for constructing inverse semigroups 
from left cancellative categories \cite{Law2,Law5} originating in the work of Leech \cite{Leech1}.

A {\em Leech category} is a left cancellative category in which any pair of arrows with a common range that can be completed to a commutative square have a pullback.
A {\em principal right ideal} in a category $C$ is a subset of the form $aC$ where $a \in C$.
A category $C$ is said to be {\em right rigid} of $aC \cap bC \neq \emptyset$ implies that $aC \subseteq bC$ or $bC \subseteq aC$;
this terminology is derived from Cohn \cite{Cohn}.
A {\em left Rees category} is a left cancellative, right rigid category in which each principal right ideal is properly contained 
in only finitely many distinct principal right ideals.
Left Rees categories are Leech categories.

With each Leech category $C$ we may associate an inverse semigroup $\mathbf{S}(C)$ as follows; all proofs may be found in \cite{Law5}.
Put
$$U = \{(a,b) \in C \times C \colon \mathbf{d}(a) = \mathbf{d}(b) \}.$$
Define a relation $\sim$ on $U$ as follows
$$(a,b) \sim (a',b') \Leftrightarrow (a,b) = (a',b')u$$
for some invertible element $u \in C$.
This is an equivalence relation on $U$ and we denote the equivalence class containing $(a,b)$ by $[a,b]$.
The product $[a,b][c,d]$ is defined as follows:
if there are no elements $x$ and $y$ such that $bx = cy$ then the product is defined to be zero;
if such elements exist choose such a pair that is a pullback.
The product is then defined to be $[ax,dy]$.
Observe that $[a,b]^{-1} = [b,a]$.
Thus $[a,b] \mathcal{L} [c,d]$ if and only if $[b,b] = [d,d]$ 
and
$[a,b] \mathcal{R} [c,d]$ if and only if $[a,a] = [c,c]$. 
The non-zero idempotents of the inverse semigroup $\mathbf{S}(C)$ are the elements of the form $[a,a]$.
Define $[a,a]^{\circ} = [\mathbf{r}(a),\mathbf{r}(a)]$.
The natural partial order is given by
$[a,b] \leq [c,d]$ if and only if $(a,b) = (c,d)p$ for some arrow $p$.

\begin{proposition} Let $C$ be a Leech category.
\begin{enumerate}

\item $\mathbf{S}(C)$ is an inverse semigroup with maximal idempotents and each $\mathcal{D}$-class contains a maximal idempotent. 

\item If the groupoid of invertible elements in $C$ is trivial then $\mathbf{S}(C)$ is combinatorial and each $\mathcal{D}$-contains exactly one maximal idempotent.

\item If $C$ is a left Rees category then $\mathbf{S}(C)$ is a Perrot semigroup (Section~4.3).

\end{enumerate}
\end{proposition}

In the case where $C$ has no non-trivial invertible elements, an equivalence class $[a,b]$ consists of the singleton $(a,b)$.
In this case, we denote the equivalence class by $ab^{-1}$.

Let $G$ be a directed graph and denote by $G^{\ast}$ the free category on $G$.

\begin{lemma} 
Free categories are left Rees categories with trivial groupoids of invertible elements.
\end{lemma}

Given a directed graph $G$, we define the \emph{graph inverse semigroup} $P_{G}$ to be the semigroup with zero defined by the above construction.
The free category has no non-trivial invertible elements and so each equivalence class is denoted by $xy^{-1}$.
Thus the non-zero elements of $P_{G}$ are of the form $uv^{-1}$ where $u,v$ are paths in $G$ with common domain.
With elements in this form, 
multiplication works in the following way:
\begin{equation*}
xy^{-1} \cdot uv^{-1} =
\begin{cases}
 xzv^{-1} & \mbox{if $u=yz$ for some path $z$}\\
 x \left( vz \right)^{-1} & \mbox{if $y=uz$ for some path $z$}\\
 0 & \mbox{otherwise.}\\
\end{cases}
\end{equation*}
 Let $xy^{-1}$ and $uv^{-1}$ be non-zero elements of $P_{G}$. 
Then
$$xy^{-1} \leq uv^{-1}
\Leftrightarrow 
\exists p \in \mathcal{G}^{\ast} \text{ such that }  
x = up \text{ and } y = vp.$$
If $xy^{-1} \leq uv^{-1}$ or $uv^{-1} \leq xy^{-1}$ then we say $xy^{-1}$ and $uv^{-1}$ are \emph{comparable}.

The following is proved in \cite{JL2} and is in fact a characterization of graph inverse semigroups.

\begin{proposition} A graph inverse semigroup is a combinatorial Perrot semigroup with maximal idempotents
such that each $\mathcal{D}$-class contains a unique maximal idempotent. 
It is in addition $E^{\ast}$-unitary.
\end{proposition}

%The proofs of the following are routine.

%\begin{lemma} Let $P_{G}$ be a graph inverse semigroup.
%\begin{enumerate}

%\item The inverse semigroup $P_{G}$ has no $0$-minimal idempotents if and only if the in-degree of each vertex is at least one.
 
%\item  The inverse semigroup $P_{G}$ has a $0$-disjunctive semilattice of idempotents 
%if and only if the in-degree of each vertex is either 0 or at least 2.

%\item The semilattice of idempotents of $P_{G}$ is pseudofinite if and only if the in-degree of each vertex is finite.

%\item The inverse semigroup $P_{G}$ is compactable if and only if $G$ has a finite number of vertices.

%\item The inverse semigroup $P_{G}$ is congruence-free if and only if $G$ is strongly connected and the in-degree of each vertex is either zero or at least two.

%\end{enumerate}
%\end{lemma}

The key result is the following and follows from our theory and an analysis that generalizes the polycyclic case.

\begin{theorem} 
Graph inverse semigroups $P_{G}$ over graphs $G$ in which each vertex has finite in-degree are pre-Boolean.
\end{theorem}

We call the Boolean completions of graph inverse semigroups {\em Cuntz-Krieger inverse semigroups} $CK_{G}$.
We may describe $CK_{G}$ in the following way:
\begin{itemize}

\item It is distributive.

\item It contains a copy of $P_{G}$ and every element of $CK_{G}$ is the join of a finite subset of $P_{G}$.

\item $e = \bigvee_{f' \in \hat{e}} f'$ for each maximal idempotent $e$ of $P_{G}$.

\item It is the freest inverse semigroup satisfying the above conditions.

\end{itemize}

A direct construction of this semigroup is given in \cite{JL1}.
The Cuntz-Krieger inverse semigroups constructed from {\em finite} directed graphs are monoids.
Their groups of units are analogues of the Thompson groups $G_{n,1}$. 

%%%%%%%%%%%%%%%%%%%%%%%%%%%%%%%%%%%%%%%%%%%%%%%%%%%%%%%%%%%%%%%%%%%%%%%%%%%%%%%%%%%%%%%%%%%%%%%%%%%%%%%%%%%%%%%%%%%%%%%%%%%%%%%%%%%%%%%%%%%%%
\section{A characterization of finite symmetric inverse monoids}

In classical Stone duality, there is an order isomorphism between the ideals of the Boolean algebra and
the open subsets of the associated Boolean space.
We generalize this result to our setting by reformulating the results of Section~7 of \cite{Lenz}. 
We shall apply this result to bring out the analogy between the finite symmetric inverse monoids $I(X)$, 
where $X$ has $n$ elements, and the finite-dimensional $C^{\ast}$-algebras $M_{n}(\mathbb{C})$.

In a groupoid $G$, define $g \mathcal{L} h$ if $g^{-1}g = h^{-1}h$, and define $g \mathcal{R} h$ if $gg^{-1} = hh^{-1}$.
Put $\mathcal{D} = \mathcal{L} \circ \mathcal{R} = \mathcal{R} \circ \mathcal{L}$,
an equivalence relation.
The $\mathcal{D}$-classes are precisely the {\em connected components} of the groupoid $G$.
A subset of a groupoid $G$ is said to be {\em invariant} if it is a union of $\mathcal{D}$-classes.

In an inverse semigroup $S$ an ideal $I \subseteq S$ is determined by the idempotents it contains $E(I)$,
since $a \in I$ if and only if $aa^{-1} \in I$ if and only if $a^{-1}a \in I$.
The set $E(I)$ is an order ideal of $E(S)$ and is {\em self-conjugate} in the sense that $e \in E(I)$ implies that $ses^{-1} \in E(I)$ for all $s \in S$. 
An ideal $T$ of $S$ is said to be {\em tightly closed} 
if $s_{1}, \ldots, s_{n} \in T$ and $\{s_{1}, \ldots, s_{n}\}$ a cover of $s$ implies that $s \in T$.
Self-conjugate order ideals of $E(S)$ are the same thing as the invariant order ideals of $E(S)$ used in \cite{Lenz};
the simple proof is left to the reader.
It follows that we may rephrase the correspondence found in Lemma~7.7 of \cite{Lenz} in terms of ideals of the
inverse semigroup rather than as order ideals of its semilattice of idempotents.
We give the proof from this point of view.

\begin{theorem} Let $S$ be a pre-Boolean inverse semigroup and $\mathsf{G}(S)$ its associated topological groupoid. 
Then there is an order isomorphism between the set of tightly closed ideals of $S$ and the set of open invariant subsets of $\mathsf{G}(S)$.
\end{theorem}
\begin{proof}
Let $T$ be a tightly closed ideal of $S$.
Define 
$$\mathbf{O}(T) = \bigcup_{t \in T} V_{t}.$$
Let $O$ be an open invariant subset of $\mathsf{G}(S)$.
Define 
$$\mathbf{C}(O) = \{ s \in S \colon V_{s} \subseteq O\}.$$
Observe that both of these functions are order-preserving.

The set $\mathbf{O}(T)$ is open by construction.
We prove that it is invariant.
Let $F \, \mathcal{D} \, G \in \mathbf{O}(T)$.
By definition, there is  an ultrafilter $A$ such that $F^{-1} \cdot F = A \cdot A^{-1}$ and $G^{-1} \cdot G = A^{-1} \cdot A$.
By definition, $G \in V_{s}$ for some $s \in T$.
Thus $s \in G$.
It follows that $s^{-1}s \in  G^{-1} \cdot G = A^{-1} \cdot A$.
Let $a \in A$ be arbitrary.
Then $b = as^{-1}s   \in A$ and  $b \in I$ since $I$ is an ideal.
Now $bb^{-1} \in A \cdot A^{-1}  = F^{-1} \cdot F$.
Let $t \in F$ be arbitrary.
Then $c = tbb^{-1} \in F$ and $c \in I$ since $I$ is an ideal.
Then $F \in V_{c}$ where $c \in I$.
It follows that $F \in \mathbf{O}(I)$, as required.

We prove that $I = \mathbf{C}(O) = \{ s \in \colon V_{s} \subseteq O\}$ is a tightly closed ideal.
Let $s \in I$ and $t \in S$.
We prove that $st \in I$; the dual result follows by symmetry.
By assumption, $V_{s} \subseteq O$.
We prove that $V_{st} \subseteq O$.
Let $F \in V_{st}$.
Then $st \in F$.
Thus $st(st)^{-1} \in F \cdot F^{-1}$.
But $st(st)^{-1} \leq ss^{-1}$ and so $ss^{-1} \in F \cdot F^{-1}$.
It follows that $A = (F \cdot F^{-1} s)^{\uparrow}$ is an ultrafilter containing $s$ and so $A \in O$.
But $A \, \mathcal{R} \, F$ and so $F \in O$ since $O$ is an invariant subset.
We have therefore shown that $I$ is an ideal.
The proof that it is tightly closed is immediate by Lemma~2.11(2).

Finally, it remains to show what happens when we iterate our two functions.
First, for every tighty closed ideal $I$ we have that $I = \mathbf{C} \mathbf{O}(I)$.
The proof of this uses the fact that the sets $V_{s}$ are compact, the definition of a tightly closed ideal and Lemma~2.11(2).
Second, for every open invariant subset $O$ we have that $O = \mathbf{O} \mathbf{C} (O)$.
The proof of this is routine.
\end{proof}

Let $S$ be an inverse $\wedge$-semigroup.
If $e$ and $f$ are non-zero idempotents define $e \preceq f$ if and only if there are elements $x_{1}, \ldots, x_{m}$ 
such that $\mathbf{d}(x_{1}),\ldots, \mathbf{d}(x_{m}) \leq f$ and $e \rightarrow \{ \mathbf{r}(x_{1}), \ldots, \mathbf{r}(x_{n}) \}$.

\begin{lemma} 
With the above definition, $\preceq$ is a preorder on $E(S) \setminus \{0\}$.
\end{lemma}
\begin{proof} We need only prove that $\preceq$ is transitive.
Let $e \preceq f$ and $f \preceq g$.
By definition there are elements $x_{1}, \ldots,x_{m}$ and $y_{1}, \ldots, y_{n}$ such that
$e \rightarrow \{\mathbf{r}(x_{1}), \ldots, \mathbf{r}(x_{m})\}$ and $\mathbf{d}(x_{i}) \leq f$ for $1 \leq i \leq m$,
and
$f \rightarrow \{\mathbf{r}(y_{1}), \ldots, \mathbf{r}(y_{n})\}$ and $\mathbf{d}(y_{j}) \leq f$ for $1 \leq j \leq n$.
It can be checked that 
$e \rightarrow \{ \mathbf{r}(x_{i}y_{j}) \colon 1 \leq i \leq m, 1 \leq j \leq n   \}$ and
$\mathbf{d}(x_{i}y_{j}) \leq g$ for all $1 \leq i \leq m, 1 \leq j \leq n$.
This shows that $e \preceq g$.
 \end{proof}

Define $\equiv$ on $E(S) \setminus \{0\}$ by $e \equiv f$ if and only if $e \preceq f$ and $f \preceq e$.
We say that an inverse $\wedge$-semigroup is {\em $0$-simplifying} if it has no non-trivial tightly closed ideals.
The following is immediate from Lemma~7.8 of \cite{Lenz} and our discussion about ideals above.

\begin{proposition} 
An inverse $\wedge$-semigroup is $0$-simplifying if and only if $\equiv$ is a universal equivalence.
\end{proposition} 

Our use of the term  $0$-simplifying was motivated by Kumjian's use of the term {\em simplification}.
Clearly, every $0$-simple semigroup is $0$-simplifying.
We now describe an example which shows that the converse is not true.

\begin{example} {\em Let $I(X)$ be a finite symmetric inverse monoid on the set $\{1, \ldots, n\}$.
This semigroup is not $0$-simple when $n \geq 2$ but we shall prove that it is $0$-simplifying.
We shall use the observation that if $g \, \mathcal{D} \, f \preceq e$ then $g \preceq e$.
Let $e$ be the partial identity on the set $\{1, \ldots, r\}$ where $1 \leq r \leq n-1$ and let 1 be the identity on the whole of $X$.
We show that $e \equiv 1$.
First $e \preceq 1$.
This can be seen by defining the $r$ partial bijections $f_{i}(i) = i$.
Second we prove that $1 \preceq e$.
This can be seen by defining the $n$ partial bijections $f_{i}(1) = i$ for $1 \leq i \leq n$.
The proof is concluded by observing that every non-zero idempotent whose rank is $r$ is
$\mathcal{D}$-related to $e$.

Alternatively, we may use Theorem~5.1 and the fact that 
the groupoid associated with $I(X)$ is just the connected groupoid $X \times X$ with the discrete topology.}
\end{example}

%%%%%%%%%%%%%%%%%%%%%%%%%%%%%%%%%%%%%%%%%%%%%%%%%%%%%%%%%%%%%%%%%%%%%%%%%%%%%%%%%%%%%%%%%%%%%%%%%%%%%%%%%%%%%%%%%%%%%%%%%%%%%%%%%%%%%%%%%%%%%%%%%%%%%%%
A  groupoid is said to be {\em principal} if all its local groups are trivial.
Such groupoids are essentially equivalence relations.
We shall now investigate the question of which Boolean inverse $\wedge$-semigroups $S$ have the property that $\mathsf{G}(S)$ are principal.
Such groupoids are interesting from the $C^{\ast}$-algebra perspective \cite{Renault}.

We recall first a construction to be found in \cite{Petrich}.
Let $S$ be an arbitrary inverse semigroup and let $F \subset E(S)$ be a subsemigroup.
Put
$$F^{c} = \{ s \in S \colon s^{-1}s,ss^{-1} \in F, \text{ and } sFs^{-1}, s^{-1}Fs \subseteq F \}.$$
Then the following is easy to verify.

\begin{lemma} With the above definition, we have the following:
\begin{enumerate}

\item $F^{c}$ is an inverse subsemigroup of $S$ whose semilattice of idempotents is $F$.

\item If $T$ is any inverse subsemigroup of $S$ whose semilattice of idempotents is $F$ then $T \subseteq F^{c}$.

\item Let $F$ be a filter in $E(S)$. Then $F^{\uparrow} \subseteq F^{c}$.
\end{enumerate}
\end{lemma}

We can now prove our first characterization.

\begin{lemma} Let $S$ be a pre-Boolean inverse $\wedge$-semigroup.
 Then $\mathsf{G}(S)$ is principal if and only if 
$F^{\uparrow} = F^{c}$
for each ultrafilter $F \subseteq E(S)$.
\end{lemma}
\begin{proof} Suppose first that the condition holds.
We prove that $\mathsf{G}(S)$ is principal. 
Let $H$ be an ultrafilter and inverse subsemigroup of $S$.
Then $F = E(H)$ is an ultrafilter in $E(S)$ such that $H = F^{\uparrow}$.
Suppose that $A$ is an ultrafilter such that $A^{-1} \cdot A = H = A \cdot A^{-1}$.
we may write $A = (aH)^{\uparrow} = (Ha)^{\uparrow}$ where $a \in A$ is arbitrary.
Clearly $a^{-1}a,aa^{-1} \in F$.
Also $aHa^{-1} \subseteq H$ and $a^{-1}Ha \subseteq H$.
It follows that $a \in F^{c}$.
By assumption $e \leq a$ where $e \in F$.
Thus $A \subseteq F^{\uparrow}$.
But $A$ is an ultrafilter and so $A = F^{\uparrow} = H$, as required.

We now assume that $\mathsf{G}(S)$ is principal and prove the condition.
Let $F$ be an arbitrary ultrafilter in $E(S)$.
Then $H = F^{\uparrow}$ is an ultrafilter in $S$.
Suppose that $a \in F^{c}$.
Then in particular $a^{-1}a,aa^{-1} \in F$.
Put $A = (aH)^{\uparrow}$.
Then $A$ is an ultrafilfter in $S$.
Observe that $A^{-1} \cdot A = H$.
Also $A \cdot A^{-1} = (aHa^{-1})^{\uparrow}$.
But $a \in F^{c}$ and so $(aHa^{-1})^{\uparrow} = H$.
By assumption $A = H$.
Thus $a \in F^{\uparrow}$, as required.
\end{proof}

We have the following useful consequence of the above result.

\begin{lemma} Let $S$ be a separative or $E^{\ast}$-unitary pre-Boolean inverse $\wedge$-semigroup.
If $\mathsf{G}(S)$ is principal then $S$ is fundamental.
\end{lemma} 
\begin{proof}
Let $a \mu a^{-1}a$.
Let $H$ be any ultrafilter containing $a^{-1}a$.
It is necessarily an inverse subsemigroup.
Put $F = E(H)$ so that $H = F^{\uparrow}$.
We claim that $a \in F^{c}$.
By assumption $a^{-1}a = aa^{-1}$.
Also if $f \in F$ then $afa^{-1} = fa^{-1}a$ and $a^{-1}fa = fa^{-1}a$.
It follows that $a \in F^{c}$ and so, by assumption, there exists $e \in F$ such that $e \leq a$.
In particular, $a \in H$.
We have shown that $V_{a^{-1}a} \subseteq V_{a}$.

Let $A$ be any ultrafilter containing $a$.
Then $A = (aA^{-1} \cdot A)^{\uparrow}$.
But $a^{-1}a \in A^{-1} \cdot A$ and $A^{-1} \cdot A$ is an ultrafilter containing $a^{-1}a$.
It follows by our result above that $a \in A^{-1} \cdot A$ and so $a^{-1} \in A^{-1} \cdot A$.
Thus $aa^{-1} = a^{-1}a \in A$.

We have proved that $a \leftrightarrow a^{-1}a$.

If $S$ is $E^{\ast}$-unitary let $F$ be any ultrafilter containing $a$.
Then it must also contain $a^{-1}a$. It follows that the idempotent $a \wedge a^{-1}a$ is non-zero.
But it lies beneath $a$ and so $a$ is an idempotent and so $a = a^{-1}a$.
It follows that $S$ is fundamental.

If $S$ is separative, we have that $a = a^{-1}a$.
It follows that $S$ is fundamental.
\end{proof}

The converse of the above lemma is not true because by Theorem~4.2 the Cuntz inverse monoid $C_{n}$ is congruence-free and so fundamental,
but its groupoid $\mathsf{G}(C_{n})$ is not principal, because there are idempotent ultrafilters that can be constructed from right-infinite
periodic strings \cite{Renault}.
The best we can do at the moment is the following which is enough for our purposes.

\begin{lemma} 
Let $S$ be a finite Boolean inverse $\wedge$-semigroup.
Then $S$ is fundamental if and only if $\mathsf{G}(S)$ is principal.
\end{lemma}
\begin{proof} Boolean inverse $\wedge$-semigroups are separative, so by Lemma~5.7 we have only one direction to prove.
Thus let $S$ be a fundamental finite Boolean inverse $\wedge$-semigroup. 
By finiteness, every ultrafilter in $S$ is principal and is generated by an element immediately above zero.
It follows that the groupoid of ultrafilters of $S$ is isomorphic to the groupoid of $0$-minimal elements $M$ of $S$.
The set $M^{0} = M \cup \{0\}$ is an ideal of $S$.
Since $S$ is fundamental so too is $M^{0}$.
But $M^{0}$ is a primitive inverse semigroup.
Such semigroups are fundamental precisely when they are combinatorial; see Exercises~II.3.10(i) of \cite{Petrich}.
It follows that $M$ is a principal groupoid.
\end{proof}

We now have a characterization of the finite symmetric inverse monoids.

\begin{theorem} The only finite Boolean inverse $\wedge$-monoids that are fundamental and $0$-simplifying are the finite symmetric inverse monoids.
\end{theorem}
\begin{proof} By Example~5.4, the finite symmetric inverse monoids are $0$-simplifying
and it is well-known, and can be verified directly, that they are inverse $\wedge$-semigroups and fundamental.
It remains to prove the converse.
Let $S$ be a finite Boolean inverse $\wedge$-monoid that is fundamental and $0$-simplifying. 
By Lemma~5.8, the groupoid $\mathsf{G}(S)$ is principal.
By finiteness, the topology is discrete.
But by Theorem~5.1, the groupoid $\mathsf{G}(S)$ can have no non-trivial open invariant subsets
and so this groupoid must consist of just one component.
It follows that $\mathsf{G}(S)$ is isomorphic to a groupoid of the form $X \times X$ where $X$ is a finite set.
The inverse monoid $S$ is isomorphic to the inverse semigroup of compact-open bisections of $\mathsf{G}(S)$ and
so is isomorphic to the inverse monoid $I(X)$.
\end{proof}

It is worth observing that in Boolean inverse $\wedge$-semigroups, the tightly closed ideals are precisely those ideals
that are also closed under joins whenever they exist.

%%%%%%%%%%%%%%%%%%%%%%%%%%%%%%%%%%%%%%%%%%%%%%%%%%%%%%%%%%%%%%%%%%%%%%%%%%%%%%%%%%%%%%%%%%%%%%%%%%%%%%%%%%%%%%%%%%%%%%%%%%%%%%%%%%%%%%%%%%%%%%%%%%%%%
\subsection{Concluding remarks }

The theory developed in this paper owes its inspiration to the work of Kellendonk \cite{Kell1,Kell2} and Paterson \cite{Pat1}
and to that of Lenz \cite{Lenz} who reconciled their two approaches.
It is worth noting that there is a precursor to our work in the theory of spectra of posets.
The earliest reference we can find that has the flavour of what we do is by B\"uchi \cite{Buchi},
but the most direct ancestor is probably Dooley \cite{D} and the Darmstadt school \cite{HK}. 
The defining property of unambiguous inverse semigroups is strongly reminiscent of the property of open balls in ultrametric spaces.
This provides the link between our work and that of Hughes \cite{Hughes1,Hughes2,Hughes3}
though he works mainly with topological groupoids and $C^{\ast}$-algebras.

%%%%%%%%%%%%%%%%%%%%%%%%%%%%%%%%%%%%%%%%%%%%%%%%%%%%%%%%%%%%%%%%%%%%%%%%%%%%%%%%%%%%%%%%%%%%%%%%%%%%%%%%%%%%%%%%%%%%%%%%%%%%%%%%%%%%%%%%%%%%%%%%%%%%%%%%%%%%%%%%%%%%

\end{document}